\documentclass{amsart}
\usepackage{amssymb,latexsym}
\usepackage{stmaryrd, comment}
\usepackage{xcolor}
\usepackage{hyperref}

\newtheorem{theorem}{Theorem}[section]
\newtheorem{lemma}[theorem]{Lemma}

\theoremstyle{definition}
\newtheorem{definition}[theorem]{Definition}

\newtheorem{proposition}[theorem]{Proposition}
\newtheorem{corollary}[theorem]{Corollary}
\newtheorem{question}[theorem]{Question}
\newtheorem{remark}[theorem]{Remark}

\numberwithin{equation}{section}

\newcommand{\D}{{\mathbb D}}

\newcommand{\C}{{\mathbb C}}

\newcommand{\N}{{\mathbb N}}

\newcommand{\T}{{\mathbb T}}

\newcommand{\HH}{\mathcal{H}}

\newcommand{\HD}{\mathcal{D}}

\makeatletter
\@namedef{subjclassname@2020}{%
  \textup{2020} Mathematics Subject Classification}
\makeatother

\begin{document}

\date{\today}

\author[C. Bellavita]{Carlo Bellavita}
\email{carlo.bellavita@gmail.com}
\email{carlo.bellavita@unimi.it}
\address{Dipartimento di Matematica ``F. Enriques''\\
  Dipartimento di Eccellenza MUR 2023-2027\\
Universit\'a degli Studi di Milano\\
Via C. Saldini 50\\
20133 Milano, Italy.}

\author[E. Dellepiane]{Eugenio Dellepiane}
\email{dellepianeeugenio@gmail.com}
\email{e.dellepiane@utwente.nl}
\address{
Department of Applied Mathematics,
University of Twente, 
7500 AE Enschede, The Netherlands}

\author[S. Luo]{Shuaibing Luo}
\email{sluo@hnu.edu.cn}
\address{School of Mathematics, Hunan University, Changsha, 410082, PR China}


\title[A characterization of equality of $\mathcal{H}(b)$ and $\mathcal{D}(\mu)$]{A characterization of equality between de Branges--Rovnyak and Dirichlet spaces}

\begin{abstract}
We characterize when a de Branges--Rovnyak space $\mathcal{H}(b)$ coincides with a
harmonically weighted Dirichlet space $\mathcal D(\mu)$. We obtain a necessary and
sufficient condition for this identification, which settles a long-standing open
problem. We illustrate the result with several examples, that include measures supported on uncountable sets.
\end{abstract}


\maketitle


\section{Introduction}

In this article, we completely characterize when a de Branges--Rovnyak space $\mathcal{H}(b)$ coincides (with equivalence of norms) with a harmonically weighted Dirichlet space $\mathcal{D}(\mu)$. These are two classes of Hilbert spaces of analytic functions on the unit disk $\mathbb{D}$ of the complex plane, that are contained in the Hardy space $H^2$. The former is parametrized by functions $b$ in $H^\infty_1,$ the closed unit ball of the Hardy space of bounded analytic functions $H^\infty$. The latter is parametrized by finite positive Borel measures $\mu$ on the unit circle $\mathbb{T}=\partial\mathbb{D}$. 

Let us recall the state of the art for this problem. The \emph{fil rouge} between the two classes of spaces has been first discovered by Sarason in 1997. In  \cite{Sa97}, he found out that when $\mu=\delta_\zeta$ is the Dirac delta centered at a point $\zeta$ in $\mathbb{T},$ and $b$ is the rational function
\begin{equation}\label{eq:Sarasonlocalrational}
b_\zeta(z)=\frac{(1-w_0)\overline{\zeta}z}{1-w_0\overline{\zeta}z}, \qquad z\in\mathbb{D},    
\end{equation}
where $w_0=(3-\sqrt{5})/2$, then the equality of sets $\mathcal{H}(b)=\mathcal{D}(\mu)$ holds, with \emph{equality} of norms, i.e., $\|\cdot\|_{\mathcal{H}(b)}\equiv\|\cdot\|_{\mathcal{D(\mu)}}$. Later, in 2010, Chevrot, Guillot and Ransford proved that the example studied by Sarason is essentially the only case with the full equality of norms. In \cite{CGR10}, they proved that $\mathcal{H}(b)=\mathcal{D}(\mu)$ with equality of norms if and only if $\mu=c\delta_\zeta (c>0)$ is a multiple of a Dirac delta anchored at some point in the boundary, and $b$ is another specific rational function, that is reminiscent of \eqref{eq:Sarasonlocalrational}, but also takes into account the number $c>0$.

When studying the equality $\mathcal{H}(b)=\mathcal{D}(\mu)$, one can assume that $b$ is a non-extreme point of the convex set $H^\infty_1$. More details will be given later, but in particular there always exists an auxiliary function $a$ in $H^\infty_1$ such that $|a|^2+|b|^2=1$ almost everywhere on $\mathbb{T}$. This function is often referred to as the \emph{Pythagorean mate} of $b$, and $(b,a)$ is called a \emph{Pythagorean pair}. In 2013, Costara and Ransford proved that whenever $a$ is a rational function, then the equality of sets $\mathcal{H}(b)=\mathcal{D}(\mu)$ holds (with \emph{equivalence} of norms) if and only if all the zeros of $a$ on the unit circle $\mathbb{T}$ are simple and $\mu$ is a measure whose support equals such boundary zero set. In particular, $\mu$ is a finite sum of Dirac deltas, generalizing the first example studied by Sarason. For the sake of completeness, there is also a third condition in the theorem of Costara--Ransford (namely, that all the zeros of $a$ lie outside of the spectrum of the inner factor of $b$). However, since it is more technical, we avoid further comments.

In 2021, Luo, Gu and Richter proved that every harmonically weighted Dirichlet space can be identified with an operator-valued de Branges--Rovnyak space, see \cite[Theorem 4.6]{LGR21}. This is a consequence of a very general and abstract result, and it involves certain properties of the shift operator on $\mathcal{D}(\mu)$.

More recently, Dellepiane, Peloso and Tabacco showed in \cite{DPT25} that the condition of Costara--Ransford can be rewritten in a language that does not require $a$ to be rational, and that an analogous statement to their theorem also holds without any assumption on the rationality of $a$. Namely, they proved that if $\mu$ is a finite sum of Dirac deltas, then $\mathcal{H}(b)=\mathcal{D}(\mu)$ if and only if
\begin{equation} \label{eq:potentialconditionstateofart}
    0<\inf_{z\in\mathbb{D}}|a(z)|^2 V_\mu(z) \leq \sup_{z\in\mathbb{D}}|a(z)|^2 V_\mu(z) <\infty,
\end{equation}
where $V_\mu\colon \mathbb{C}\to[0,+\infty]$ is a potential associated to the measure $\mu$,
\begin{equation}\label{eq:potentialintro}
V_\mu(z)=\int_{\mathbb{T}}\frac{\operatorname{d}\!\mu(\xi)}{|z-\xi|^2},\qquad z\in\mathbb{C}.
\end{equation}

In parallel, Bellavita, Dellepiane, Hartmann and Mashreghi used the same condition \eqref{eq:potentialconditionstateofart} on the interplay between $a$ and $V_\mu$ to construct an example of equality $\mathcal{H}(b)=\mathcal{D}(\mu)$, where $\mu$ has an infinite support. In particular, they proved that whenever $b=(1+u)/2$ for a one-component inner function $u$, then $\mathcal{H}(b)=\mathcal{D}(\mu)$ if and only if \eqref{eq:potentialconditionstateofart} holds. This forces $\mu$ to have the same support as the Clark measure of the one-component inner function $u$. In particular, $\mu$ is a (possibly) infinite sum of isolated atoms.



All these results that we listed, that span across the last three decades, leave open the general problem of characterizing, for arbitrary symbols
$b$ in $H^\infty_1$ and measures $\mu$, when
$\mathcal{H}(b)=\mathcal D(\mu)$. In this paper we give a necessary and sufficient condition for this identification.

The first issue in this problem is the dependence of \(\mathcal{H}(b)\) on the inner factor of~\(b\). The following question was posed in \cite[Question 6.3]{CR13}.

\begin{question}\label{cosrandques}
Let $b = b_i b_o$ be the inner-outer factorization of $b \in H^\infty_1$. Does $\mathcal D(\mu) = \mathcal{H}(b)$ imply that $\mathcal D(\mu) = \HH(b_o)$?
\end{question}

Our first result gives an affirmative answer, by establishing a stronger structural property of the Pythagorean pair  \((b,a)\). We say that \((b,a)\) is a Corona pair if
\[
\inf_{z\in\mathbb D}\bigl(|a(z)|+|b(z)|\bigr)>0.
\]
\begin{theorem}\label{T:main1}
    If $\mathcal{H}(b)=\mathcal{D}(\mu)$, then $b$ is non-extreme and the Pythagorean pair $(b,a)$ is a Corona pair. In particular, it follows that $\mathcal{H}(b)=\HH(b_o),$ where $b_o$ is the outer factor of $b$.
\end{theorem}


The Corona condition appears naturally in the theory of de Branges--Rovnyak spaces. In particular, when the Pythagorean pair $(b,a)$ satisfies it, the space $\mathcal{H}(b)$ has a much more accessible description: namely, $\mathcal{H}(b)$ coincides with the range of the co-analytic Toeplitz operator~$T_{\bar{a}}$. The fact that the Corona condition is \emph{necessary} for the equality $\mathcal{H}(b)=\mathcal D(\mu)$ is fundamental in this work, and it led us to our main result. 

\begin{theorem}\label{T:main2}
Consider a function $b$ in $H^\infty_1$ and a positive finite Borel measure $\mu$ on $\mathbb{T}$. Then, the de Branges--Rovnyak space $\mathcal{H}(b)$ coincides with the harmonically weighted Dirichlet space $\mathcal{D}(\mu)$ (with equivalence of norms) if and only if the two following conditions are satisfied:
\begin{enumerate}
\item[(i)] $b$ is a non-extreme function and the Pythagorean pair $(b,a)$ satisfies the Corona condition;
\item[(ii)] The potential condition
    \begin{equation}\label{suffdmueqmba}
        c\le |a(z)|^2(1+V_\mu(z))\le C,
        \qquad z\in\mathbb D,
    \end{equation}
holds, where \(c, C > 0\) are positive constants.
\end{enumerate}
\end{theorem}

Condition \eqref{suffdmueqmba} can equivalently be written as
\[
|a(z)|\asymp\bigl(1+V_\mu(z)\bigr)^{-1/2},
\qquad z\in\mathbb D.
\]
It therefore prescribes how the decay of the Pythagorean mate $a$ must compensate for the growth of the potential $V_\mu$. Together with the Corona condition, this estimate gives a characterization for the equality $\mathcal{H}(b)=\mathcal{D}(\mu)$, valid for arbitrary pairs \((b,\mu)\) of $H^\infty_1$-functions and finite positive Borel measures on the unit circle. 

We remark that the proof of Theorem \ref{T:main2} is not a generalization of previous proofs of other results in this line of research. It relies on the countable Corona theorem, which beautifully allows to overcome many technical obstacles. 

The rest of the paper is organized as follows. Section \ref{S:preliminaries} contains the main definitions and preliminaries. Sections \ref{S2} and \ref{S3} contain the proofs of Theorems \ref{T:main1} and \ref{T:main2}, respectively. Section \ref{S4}
collects some remarks. The remaining three sections are devoted to examples. Notably, we show that the equality $\mathcal{H}(b)=\mathcal{D}(\mu)$ can hold also for (singular) measures supported on \emph{uncountable} sets. In Section \ref{S5}, we show that if $\mu$ is in the class of $d$\emph{-Ahlfors-David regular} measures $(0<d<1)$, then $\mathcal{D}(\mu)$ coincides with a de Branges--Rovnyak space. This class of measure includes the Cantor measure, whose support is uncountable and has no atoms. In Section \ref{S6} we consider purely atomic measures, with isolated atoms, whose support may nevertheless be uncountable. Finally, in Section \ref{S7}, we consider measures with atoms at accumulation points of other atoms. In each case we show that the associated Dirichlet
space can coincide with a de Branges--Rovnyak space.

\section{Preliminaries}\label{S:preliminaries}

This paper deals with classes of subspaces of the Hardy space $H^2$. Here, we gather the main definitions. The Hardy space $H^2$ is the space of all analytic functions $f$ on $\mathbb{D}$ such that
\[
\|f\|_{H^2}^2=\sup_{0<r<1}\int_{\mathbb{T}} |f(r\zeta)|^2 \operatorname{d}\!m(\zeta) <\infty,
\]
where $m$ is the normalized Lebesgue on the unit circle $\mathbb{T}$. It is a Hilbert space that enjoys many beautiful and important properties, as a consequence of several classical results from complex analysis, harmonic analysis and operator theory. For example, it is known that every function $f$ in $H^2$ can be decomposed as the product of an \emph{inner} function and an \emph{outer} function. The outer factor $f_o$ captures the behaviour of $f$ on the boundary $\mathbb{T}$, and it admits the explicit expression
\[
f_o(z)=\exp{\left(\int_{\mathbb{T}}\frac{\zeta+z}{\zeta-z}\log|f(\zeta)|\operatorname{d}\!m(\zeta)\right)},\qquad z\in\mathbb{D}.
\]
The function $f_o$ never vanishes on $\mathbb{D}$, and $f_i:=f/f_o$ is an inner function, that is, $f_i$ is bounded and $|f_i|=1$ $m$-a.e. on $\mathbb{T}$. The inner factor $f_i$ retains information on the behaviour of $f$ on $\mathbb{D}$, e.g., the location of its zeros.

Finally, we recall the Corona theorem, a classic result by Carleson on the structure of maximal ideals of $H^\infty$. It states that, given functions $f_1,\ldots,f_N$ in $H^\infty$, then there exist functions $g_1,\ldots,g_N$ in $H^\infty$ such that $\sum_{j=1}^N f_j g_j=1$
if and only if
\[
\inf_{z\in\mathbb{D}} \sum_{j=1}^N|f_j(z)| >0.
\]
  In Section \ref{S3} we will also discuss a Corona theorem for countably many functions. For more details on Hardy spaces, a standard reference is \cite{Duren}.

\subsection{De Branges--Rovnyak spaces}
For $b$ in $H^\infty_1$, the \emph{de Branges--Rovnyak space} $\mathcal{H}(b)$ is the
reproducing kernel Hilbert space of analytic functions on $\mathbb D$ with kernel
\[
    k_w^b(z)=\frac{1-\overline{b(w)}\,b(z)}{1-\overline wz},\qquad z,w\in\mathbb D .
\]

The space $\mathcal{H}(b)$ also can be viewed as the range of the square-root operator $T=(I-T_bT_b^*)^\frac{1}{2}$, that is well-defined on $H^2$. In this case,
\[
\|Tf\|_{\mathcal{H}(b)}=\|f\|_{H^2}, \qquad f\in\operatorname{Ker}(T)^\bot.
\]
It can be proved that $\mathcal{H}(b)$ is always contained in $H^2$. When $b$ is an inner function, $\HH(b)$ is a closed subspace of $H^2$ with the $H^2$-norm, and it is called the \emph{model space}, see \cite{GMR16}.

These spaces can be very difficult to treat from a functional point of view. The structure of $\mathcal{H}(b)$ heavily depends on the properties of the function $b$, and in general de Branges--Rovnyak spaces may display very different behaviors. One nice approach to the theory of $\mathcal{H}(b)$ is through the notion of extremality. Recall that \(b\) is non-extreme if
\(
\log(1-|b|^2)\in L^1(\mathbb T,m)
\). The non-extremality of $b$ is equivalent to the fact that polynomials are dense in $\mathcal{H}(b)$. This explains why, whenever $\mathcal{H}(b)=\mathcal{D}(\mu),$ then necessarily $b$ is non-extreme.

When $b$ is non-extreme there exists a unique outer function \(a\), normalized by \(a(0)>0\), such that
\[
|a|^2+|b|^2=1
\qquad m\text{-almost everywhere on }\mathbb T.
\]
The function \(a\) is called the Pythagorean mate of \(b\), and we refer to $(b,a)$ as a Pythagorean pair.

The Corona condition appears naturally in the theory of de Branges--Rovnyak spaces. In particular, when the Pythagorean pair $(b,a)$ satisfies it, the space $\mathcal{H}(b)$ has a much more accessible description~\cite[Chapter 28]{FM16}. Indeed, in this case, $\mathcal{H}(b)$ coincides with the range of the co-analytic Toeplitz operator $T_{\bar{a}}$, that is,
\[
    \mathcal{H}(b)=\mathcal M(\bar a):=\{f\in H^2:\ f=T_{\bar a}g\ \text{for some }g\in H^2\},
\]
with equivalent norms, where
\(
    \|T_{\bar a}g\|_{\mathcal M(\bar a)}=\|g\|_{H^2}
\) (the operator $T_{\bar{a}}$ is injective, since $a$ is outer). For more details on de Branges--Rovnyak spaces, we refer to \cite{FM16, Sa94}.

\subsection{Harmonically weighted Dirichlet spaces}
 Given a finite positive Borel measure $\mu$ on $\mathbb T$, let
\[
    P_\mu(z)=\int_{\mathbb T}\frac{1-|z|^2}{|\xi-z|^2}\,\operatorname{d}\!\mu(\xi),\qquad z\in\mathbb D,
\]
be its Poisson integral. Notice that $P_\mu$ is related to the potential $V_\mu$ defined in \eqref{eq:potentialintro} by
\[
P_\mu(z)=(1-|z|^2)V_\mu(z).
\]
The \emph{harmonically weighted Dirichlet space}
$\mathcal D(\mu)$, introduced in~\cite{ Al93, Ri91}, consists of all $f\in H^2$ such that
\[
     D_\mu(f)=\int_{\mathbb D}|f'(z)|^2P_\mu(z)\operatorname{d}\!A(z)<\infty,
\]
where $\operatorname{d}\!A$ is the normalized area measure, with the norm
\[
    \|f\|_{\mathcal D(\mu)}^2=\|f\|_{H^2}^2+ D_\mu(f).
\]
In general, $\mathcal{D}(\mu)$ contains all polynomials, it is contained in $H^2$, but it is not closed in the $H^2$-norm. See \cite[Chapter 7]{EKMR14} for an exhaustive survey. 

The local Dirichlet integral $D_{\xi}(f)$ at $\xi \in \T$ is defined by
\begin{equation*}
D_{\xi}(f) = \|Q_\xi f\|_{H^2}^2, \quad \text{ where } \quad Q_\xi f(z) = \frac{f(z) - f(\xi)}{z - \xi}, \quad z\in\mathbb{D}.
\end{equation*}
It is well-known~\cite[Proposition 2.2]{RS91} that
\begin{align*}
\int_{\mathbb{D}} |f'(z)|^2 P_\mu(z) \operatorname{d}\!A(z) = \int_\T D_\xi (f) \operatorname{d}\!\mu(\xi),
\end{align*}
and for every $f \in \mathcal{D}(\mu)$, the radial value $f(\xi)$ exists $\mu$-almost everywhere, and
\begin{equation*}
\|f\|_{\mathcal{D}(\mu)}^2 = \|f\|_{H^2}^2 + \int_{\mathbb{T}} \|Q_{\xi} f\|_{H^2}^2 \operatorname{d}\!\mu(\xi).
\end{equation*}


\section{Proof of Theorem \ref{T:main1}}\label{S2}

To prove Theorem \ref{T:main1}, we need two preliminary lemmas. For $b \in H^\infty$ and $r> 0$, we consider the sublevel set 
\[
\Omega(b,r)=\{w\in\mathbb{D}\colon |b(w)|< r\}.
\]
\begin{lemma}\label{L:sufcondabcorona}
Let \((b, a)\) be a Pythagorean pair. Suppose that there exists some \(0 < \delta \le \sqrt{2}/2\) such that
\begin{equation}\label{assumpbovera}
L:=\max\left\{1, \sup_{z\in\Omega(b,\delta)}\left\vert \frac{b(z)}{a(z)}\right\vert \right\}  <\infty.
\end{equation}
Then,
\[
\inf_{z\in\Omega(b,\delta)}|a(z)| \ge \delta/L,
\]
and in particular $(b,a)$ forms a Corona pair.
\end{lemma}

\begin{proof}
Set 
\(
h(z) = -\log|a(z)| \ge 0,\) and \(M = \log(L/\delta)\geq 0.
\)
Since \(a\) is in $H^\infty$ and outer, $h$ belongs to $L^1(\mathbb{T},m)$ and
\begin{equation*}
h(z) = P[h\vert_{\mathbb{T}}](z) = \int_\T \frac{1-|z|^2}{|\xi - z|^2} \,h(\xi) \operatorname{d}\!m(\xi), \qquad z \in \D.
\end{equation*}
At each point  \(z\in\partial\Omega(b,\delta) \cap \mathbb{D}\), by continuity we have that $|b(z)|=\delta$ and then by \eqref{assumpbovera} we also have that \(|a(z)| \ge \delta/L\). In particular, \(h(z) \le M\) on $\partial\Omega(b,\delta) \cap \mathbb{D}$. The gluing lemma for subharmonic functions~\cite[Theorem 2.4.5]{Ra95} shows that
\[
v(z) = \begin{cases}
\max\{h(z) - M, 0\}, & z \in \Omega(b,\delta), \\
0, & z \in \mathbb{D} \setminus \Omega(b,\delta),
\end{cases}
\]
is subharmonic on \(\mathbb{D}\), and it satisfies \(0 \le v(z) \le h(z), z\in\mathbb{D}\).

We claim that \(v(r\xi) \to 0\) for \(m\)-almost every \(\xi \in \mathbb{T}\). If \(|b(\xi)| > \delta\), the points \(r\xi\) eventually lie outside \(\Omega\).
If \(|b(\xi)| \le \delta\), then since $0\leq\delta\leq\sqrt{2}/2,$
\[
|a(\xi)| \ge \sqrt{1 - \delta^2} \ge \delta \ge\frac{\delta}{L},
\]
so \(h(\xi) \le M\), which again implies \(v(r\xi) \to 0\). We want to prove that
\begin{equation}\label{eq:intvrto0}
\lim_{r\to 1}\int_{\mathbb{T}} v(r\xi)\operatorname{d}\!m(\xi) = 0.    
\end{equation}
We do the following trick. Since $h \in L^1$, the radial dilations \((h(r\cdot))_{0<r<1}\) converge to \(h\) in \(L^1\) as $r\to 1$. Since $0\leq v\leq h$ on the unit disk,
\[
v(r\xi)\leq h(r\xi)-h(\xi)+h(\xi)\leq h(\xi)+|h(r\xi)-h(\xi)|.
\]
In particular, it holds the pointwise estimate
\[
v(r\xi)\leq \min\{v(r\xi), h(\xi)\}+ |h(r\xi)-h(\xi)|, \qquad \text{a.e.}\,\,\xi\in\mathbb{T}.
\]
Then,
\[
0\leq \int_{\mathbb{T}} v(r\xi)\operatorname{d}\!m(\xi) \leq \int_{\mathbb{T}} \min\{v(r\xi),h(\xi)\}\operatorname{d}\!m(\xi)+\|h(r\cdot)-h\|_{L^1}.
\]
Since $\min\{v(r\xi),h(\xi)\}\leq v(r\xi)\to 0$ and $\min\{v(r\xi),h(\xi)\}\leq h(\xi)$ almost everywhere, and $h\in L^1,$ by dominated convergence we have that
\[
\lim_{r\to 1}\int_{\mathbb{T}} \min\{v(r\xi),h(\xi)\}\operatorname{d}\!m(\xi)=0.
\]
Now, \eqref{eq:intvrto0} follows from the convergence $h(r\cdot)\to h$ in $L^1$.

For fixed \(z \in \mathbb{D}\) and \(|z| < r < 1\), the subharmonic Poisson inequality~\cite[Theorem 2.4.1]{Ra95} gives
\[
0 \le v(z) \le \int_{\mathbb{T}} \frac{r^2 - |z|^2}{|r\xi - z|^2} \, v(r\xi) \operatorname{d}\!m(\xi) \le \frac{r + |z|}{r - |z|} \int_{\mathbb{T}} v(r\xi)  \operatorname{d}\!m(\xi).
\]
Letting $r \to 1$ proves $v \equiv 0$. Hence $h \le M$ on $\Omega(b,\delta)$, which is equivalent to
\[
|a(z)|\geq\frac{\delta}{L}, \qquad z\in\Omega(b,\delta). 
\]
Outside $\Omega(b,\delta)$ we have that $|b| \ge \delta$. The two conditions together implies the Corona condition
\[
\inf_\mathbb{D} |b|+|a|>0. \qedhere
\]
\end{proof}

A refinement of the argument used in~\cite[Theorem 3.3]{DPT25} and in \cite[Theorem 1.4]{BD25} give the following result. 

\begin{lemma}\label{L:vmubdsublev}
Consider $b\in H^\infty_1$ (not necessarily non-extreme). If $\mathcal{H}(b) \subseteq \mathcal{D}(\mu)$, then the potential $V_\mu$ is bounded on every (non-trivial) sublevel set of $b$, that is, for every $r<1$
\[
\sup_{\Omega(b,r)} V_\mu< \infty.
\]
\end{lemma}
\begin{proof}
Choose $C > 0$ so that
\[
D_\mu(f) \leq C \|f\|_{\mathcal{H}(b)}^2,\qquad f\in\mathcal{H}(b).
\]
Consider the partition
\begin{align*}
    \Omega(b,r)&=\{w\in\mathbb{D}\colon |w|< \frac{r+1}{2},|b(w)|< r\}\cup \{w\in\mathbb{D}\colon |w|\geq \frac{r+1}{2},|b(w)|< r\}\\
    &=:\Omega_1 \cup \Omega_2.
\end{align*}
Since $\Omega_1\subseteq B(0,\tfrac{r+1}{2})$, one can easily check that $\sup_{\Omega_1}V_\mu<\infty,$ for every $r<1$.

Take $w\in \Omega_2$.
Using the same argument as in the proof of~\cite[Theorem 3.3]{DPT25}, for every kernel $k_w^b,w\in\mathbb{D},$
\begin{align*}
    C\|k_{w}^b\|_{\mathcal{H}(b)}^2&\geq    D_\mu(k_{w}^b)=  \int_{\mathbb{T}} D_\zeta(k_{w}^b) \,  \operatorname{d}\!\mu(\zeta) \\
    &\geq   \int_{\mathbb{T}} \|k_{w}^b\|_{\mathcal{H}(b)}^2 \frac{\big(1-|b(w)|\big)^2 \big(|w| - |b(w)|\big)^2}{|\zeta-w|^2 \big(1-|b(w)|^2\big)} \, \operatorname{d}\!\mu(\zeta) \\
    &=   \|k_{w}^b\|^2_{\mathcal{H}(b)} \frac{\big(1-|b(w)|\big)\big(|w| - |b(w)|\big)^2}{1+|b(w)|} \int_{\mathbb{T}} \frac{1}{|\zeta -w|^2}  \, \operatorname{d}\!\mu(\zeta) .
\end{align*}
In particular, for $w\in\Omega_2$,
\begin{align*}
    C&\geq \frac{\big(1-|b(w)|\big)\big(|w| - |b(w)|\big)^2}{1+|b(w)|} V_\mu(w)\\
    &\geq \frac{\big(1-r\big)\big(\tfrac{r+1}{2} - r\big)^2}{2} V_\mu(w) = \frac{(1-r)^3}{8} V_\mu(w).
\end{align*}
It follows that
\[
\sup_{w\in\Omega_2}V_\mu(w) \leq \frac{8C}{(1-r)^3}<\infty,
\]
and the proof is complete.
\end{proof}

We come to the main proof.
\begin{proof}[Proof of Theorem \ref{T:main1}] By~\cite[Theorem 3.3]{CR13}, we have that
\begin{align}\label{dmueqhbnecco}
1 + V_\mu(z) \asymp 1 + \left\vert \frac{b(z)}{a(z)}\right\vert^2,\qquad z\in\mathbb{D}.
\end{align}
For fixed $r< \sqrt{2}/2$, Lemma \ref{L:vmubdsublev} shows that $V_\mu$ is uniformly bounded on the sublevel set $\Omega(b,r)$. It follows from (\ref{dmueqhbnecco}) that
\[
L:=\max\left\{1,\sup_{\Omega(b,r)}\left\vert \frac{b}{a}\right\vert\right\}<\infty,
\]
and then by Lemma \ref{L:sufcondabcorona} $(b,a)$ forms a Corona pair. To conclude, we prove that $\mathcal{H}(b)=\HH(b_o)$. Consider the inner-outer factorization $b=b_ib_o$. Then,
\[
|b_o(z)|+|a(z)| \geq |b(z)|+|a(z)|, \qquad z\in\mathbb{D},
\]
hence $(b_o,a)$ is a Corona pair as well. In particular, $\mathcal{H}(b)=\mathcal{M}(\bar{a})=\HH(b_o)$, and the proof is complete. 

\end{proof}

\subsection{Another proof for the equality $\mathcal{H}(b)=\HH(b_o)$}

We have proved that, if $\mathcal{H}(b)=\mathcal{D}(\mu)$, then $(b,a)$ forms a Corona pair, and in particular the de Branges--Rovnyak space $\mathcal{H}(b)$ only depends on the outer factor. We give another proof of this fact, that does not use the Corona property, but uses the theory of model spaces. 

\begin{lemma}\label{innerdmu}
Let $u$ be an inner function belonging to $\mathcal{D}(\mu)$. Set the model space
$
K_u = H^2 \ominus uH^2
$~\cite{GMR16}.
For $f\in K_u$, let 
\[
C_u f(\zeta) = u(\zeta)\overline{\zeta f(\zeta)} \qquad \text{ for almost every } \zeta\in \mathbb{T}
\]
be the canonical conjugation~\cite{GP06, GP07}.
If $f,C_u f\in \mathcal{D}(\mu)$, then $uf\in \mathcal{D}(\mu)$ and
\[
D_\mu(uf) \le 8\big(D_\mu(f) + D_\mu(z C_u f)\big).
\]
\end{lemma}

\begin{proof}
For $m$-almost every $\lambda,\zeta\in\mathbb{T}$, we have that
\begin{align*}
|u(\zeta)f(\zeta) - u(\lambda)f(\lambda)| &= |u(\zeta)f(\zeta) - u(\zeta) f(\lambda) +u(\zeta) f(\lambda) - u(\lambda)f(\lambda)|\\
 &\le |f(\zeta) - f(\lambda)| + | f(\lambda)(u(\zeta)-u(\lambda))|.
\end{align*}
Now, since $\zeta C_uf(\zeta)=u(\zeta)\overline{f(\zeta)}$ for $m$-almost every $\zeta \in \mathbb{T}$,
\begin{align*}
   \overline{f(\lambda)} (u(\lambda)-u(\zeta))  &= \lambda C_u f(\lambda) - \zeta C_u f(\zeta) + u(\zeta)(\overline{f(\zeta)}-\overline{f(\lambda)}),
\end{align*}
and we have that
\begin{equation*}
    | f(\lambda)(u(\zeta)-u(\lambda))|\leq |\lambda C_u f(\lambda) - \zeta C_u f(\zeta)| + |f(\zeta) - f(\lambda)|.
\end{equation*}
We found that
\begin{align*}
|u(\zeta)f(\zeta) - u(\lambda)f(\lambda)| 
 &\le 2|f(\zeta) - f(\lambda)| + |\lambda C_u f(\lambda) - \zeta C_u f(\zeta)|.
\end{align*}
Thus,
\begin{align*}
    D_{\lambda}(uf)=\int_{\mathbb{T}}\left|\frac{u(\zeta)f(\zeta) - u(\lambda)f(\lambda)}{\zeta-\lambda}\right|^2\operatorname{d}\!m(\zeta)\leq 8D_\lambda(f)+2D_\lambda(zC_uf),
\end{align*}
and we conclude that 
\[
D_\mu(uf) \le 8D_\mu(f) + 2D_\mu(zC_u f)<\infty. \qedhere
\]

\end{proof}

We can give a second proof of the following result:
\begin{theorem}\label{dmuhbbouter}
Let $b \in H^\infty_1$ and let
$
b = b_i b_o
$
be its inner-outer factorization. If $\mathcal{D}(\mu) = \mathcal{H}(b)$, then $\mathcal{D}(\mu) = \mathcal{H}(b_o)$.
\end{theorem}

\begin{proof}
Set $u = b_i$. Then
\[
\mathcal{H}(b) = K_u \oplus u\mathcal{H}(b_o),
\]
see~\cite{FHR18}.
Thus $K_u \subseteq \mathcal{D}(\mu)$.
Since $\mathcal{D}(\mu)$ contains all polynomials, $b$ is non-extreme. The equality $|b_o| = |b|$ $m$-almost everywhere implies that $b_o$ is also non-extreme. Hence $1\in \mathcal{H}(b_o)$, so
\[
u \in u\mathcal{H}(b_o) \subseteq \mathcal{D}(\mu).
\]

Let $f\in K_u$. Then $C_u f\in K_u$. So $f, C_u f\in \mathcal{D}(\mu)$. Lemma \ref{innerdmu} gives $uf\in \mathcal{D}(\mu)$. Thus
\[
uK_u \subseteq \mathcal{D}(\mu) = \mathcal{H}(b).
\]
Under the hypothesis $\mathcal{D}(\mu) = \mathcal{H}(b)$, \cite[Theorem 5.4(iii)]{CR13} states that
\[
\mathcal{H}(b) = \mathcal{H}(b_o) \iff uK_u \subseteq \mathcal{H}(b).
\]
The required inclusion has just been established. Therefore $\mathcal{D}(\mu) = \mathcal{H}(b_o)$.

\end{proof}

\section{Proof of Theorem \ref{T:main2}}\label{S3}
As we mentioned in the Introduction, the proof of Theorem \ref{T:main2} relies on the countable Corona theorem. We recall some basic facts of $\ell^2$-valued Hardy spaces. We begin by introducing
 \[
 H^2(\ell^2)=\{ F=(F_n)_{n\in\mathbb{N}}\colon F_n\in H^2,\|F\|_{H^2(\ell^2)}^2:=\sum_{n=0}^\infty \|F_n\|_{H^2}^2<\infty\}.
 \]
 We can also view the elements of $H^2(\ell^2)$ as (analytic) functions $F\colon \mathbb{D}\to\ell^2$, but we prefer the interpretation of infinite vectors $(F_n)_n$.
We also consider 
\[
H^\infty(\ell^2)=\{ F=(F_n)_{n\in\mathbb{N}}\colon F_n\in H^\infty, \,\,\sup_{z\in\mathbb{D}}\sum_{n=0}^\infty|F_n(z)|^2<\infty\}.
\] 
One can check that $H^\infty(\ell^2)\subseteq H^2(\ell^2)$.


Given a vector $F\in H^\infty(\ell^2)$, we consider the row multiplier $M_F^R : H^2(\ell^2)\to H^2$ and the column multiplier $M_F^C : H^2 \rightarrow H^2(\ell^2)$, that is,
\[
M_F^R (h_0, h_1, \ldots)^T = \sum_{n=0}^\infty F_n h_n, \qquad M_F^C h = (F_0h, F_1h, \ldots)^T.
\]
If $F\in H^\infty(\ell^2)$, then $M_F^R,M_F^C$ are bounded. 
The adjoint of the row multiplier is given by
$(M_F^R)^*h=(T_{\overline{F_n}}h)_{n\in \N}$.
Indeed, for $(k_n)_{n\in \N}\in H^2(\ell^2)$ and $h\in H^2$,
\[
\big\langle M_F^R(k_n),h\big\rangle_{H^2}
=\sum_{n=0}^\infty\langle F_nk_n,h\rangle_{H^2}
=\sum_{n=0}^\infty\langle k_n,T_{\overline{F_n}}h\rangle_{H^2}
=\big\langle(k_n)_n,(T_{\overline{F_n}}h)_n
\big\rangle_{H^2(\ell^2)}.
\]

In the setting of $H^\infty(\ell^2)$, a countable version of the Corona theorem is available, by a result of Rosenblum \cite{Ro80}. If $F\in H^\infty(\ell^2)$, then
\[
\inf_{z\in\mathbb{D}}\sum_{n=0}^\infty |F_n(z)|^2>0
\]
if and only if there exists $G=(G_n)_n$ in $H^\infty(\ell^2)$ such that $\sum_{n=0}^\infty F_n G_n =1.$

These constructions are fundamental for the proof of the main theorem.

\begin{proof}[Proof of Theorem \ref{T:main2}]
The necessity follows from Theorem \ref{T:main1} and \eqref{dmueqhbnecco}. We deal with the sufficiency. Because of (i), the spaces $\mathcal{H}(b)$ and $\mathcal{M}(\overline{a})$ coincide. Then, it suffices to verify that $\mathcal{D}(\mu)=\mathcal{M}(\overline{a})$, assuming the potential condition.

The space \(L^2(\mu)\) is separable, since $\mu$ is finite. Choose a countable orthonormal basis \((e_n)_{n=1}^N\), where \(N\) may be finite or infinite. We use the notation $1\leq n<N+1$ to include all such vectors. Using Parseval's identity, we can rewrite the potential $V_\mu$ as follows. For $z\in\mathbb{D}$,
\begin{align}
\nonumber V_\mu(z)&= \int_{\mathbb{T}}\frac{1}{|1-\overline{\xi}z|^2}\operatorname{d}\!\mu(\xi) = \|c_{\overline{z}}\|_{L^2(\mu)}^2 \\
\label{eq:formulaV}&= \sum_{n=1}^N |\langle e_n,c_{\overline{z}}\rangle_{L^2(\mu)}|^2 = \sum_{n=1}^N \left|\int_{\mathbb{T}}\frac{e_n(\xi)}{1-\overline{\xi}z}\operatorname{d}\!\mu(\xi)\right|^2.
\end{align}

Define
\begin{equation*}
F_0(z) = a(z), \qquad F_n(z) = za(z) \int_{\mathbb{T}} \frac{e_n(\xi)}{1 - \overline{\xi} z} \, \operatorname{d}\!\mu(\xi), \quad 1 \le n < N+1.
\end{equation*}
The functions $F_n$ are analytic in $\D$. Also, for fixed $1\leq n< N+1$ the function $F_n$ is bounded: 
\[
|F_n(z)|\leq |a(z)|\|e_n\|_{L^2(\mu)} V_\mu(z)^\frac{1}{2} \leq \sup_{\mathbb{D}}|a|V_\mu^\frac{1}{2} <\infty.
\]
Notice that
\begin{equation*}
1 + |z|^2 V_\mu(z) \asymp 1 + V_\mu(z), \qquad z \in \mathbb{D},
\end{equation*}
and, by \eqref{eq:formulaV},
\begin{align*}
\sum_{n=0}^N |F_n(z)|^2 &= |a(z)|^2 + |za(z)|^2  V_\mu(z) = |a(z)|^2\big(1 + |z|^2 V_\mu(z)\big) \asymp 1, \qquad z \in \mathbb{D}.
\end{align*}
Thus $F=(F_n)_{n=0}^N$ belongs to $H^\infty(\ell^2)$, with its squared norm bounded above and below throughout $\mathbb{D}$.
Thus by the countable Corona theorem \cite{Ro80}, there is $G=(G_n)_{n=0}^N\in H^\infty(\ell^2)$ such that
\[
\sum_{n=0}^N F_n(z) G_n(z) = 1, \qquad z \in \mathbb{D}.
\]
Then the row multiplier $M_F^R$ has bounded right inverse $M_G^C$, that is, $M_F^R M_G^C$ is the identity operator on $H^2$. Taking adjoints gives
\[
h=(M_G^C)^*(M_F^R)^*h=(M_G^C)^*\bigl(T_{\overline{F_n}} h\bigr)_{n}, \qquad h\in H^2.
\]
In particular,
\begin{align*}
\|h\|_{H^2}
&=\bigl\|(M_G^C)^*(M_F^R)^*h\bigr\|_{H^2}\le \|M_G^C\|\,\|(M_F^R)^*h\|_{H^2(\ell^2)}
\le \|M_G^C\|\,\|M_F^R\|\,\|h\|_{H^2},
\end{align*}
and
\[
\frac{1}{\|M_G^C\|^2}\|h\|_{H^2}^2
\le
\sum_{n=0}^N\|T_{\overline{F_n}} h\|_{H^2}^2
\le
\|M_F^R\|^2\|h\|_{H^2}^2,
\]
since
\[
\|(M_F^R)^*h\|_{H^2(\ell^2)}^2
=\sum_{n=0}^N\|T_{\overline{F_n}} h\|_{H^2}^2.
\]
We want to compute the norm $\|T_{\overline{F_n}} h\|_{H^2}$, for every polynomial $h$.
By the definition of \(F_n\), for $n\ge 1$,
\[
F_n(z)
=za(z)\int_{\mathbb{T}}\frac{e_n(\xi)}{1-\overline{\xi}z}\operatorname{d}\!\mu(\xi)=\sum_{k=1}^{\infty}
\left(
  \int_{\mathbb T}e_n(\xi)\overline\xi^{\,k-1}\,\operatorname{d}\!\mu(\xi)
\right)z^k a(z).
\]
We set
\[
\gamma_k(n):=\int_{\mathbb T}e_n(\xi)\overline\xi^{\,k-1}\,\operatorname{d}\!\mu(\xi), \qquad n,k\geq 1.
\]
If $h$ has degree $d$, then $T_{\overline{F_n}}h,T_{\bar{a}}h$ are also polynomials of degree at most $d$. Let us write \(
f=T_{\overline a}h=\sum_{m=0}^{d}c_mz^m
\). For $j\leq d$,
\begin{align*} \label{Eq:coeffzj}
\langle T_{\overline{F_n}}h,z^j\rangle_{H^2}&=\langle h,a(z)z^j\sum_{k=1}^\infty \gamma_k(n)z^k\rangle_{H^2}=\langle h,a(z)z^j\sum_{k=1}^d \gamma_k(n)z^k\rangle_{H^2}\\
&=\langle T_{\bar{a}}h,\sum_{k=j+1}^{d+j} \gamma_{k-j}(n)z^k\rangle_{H^2}
=
\sum_{k=j+1}^d c_k \overline{\gamma_{k-j}(n)}.    
\end{align*}
It follows that
\begin{align*}
T_{\overline{F_n}}h(z)
&=
\sum_{j=0}^{d-1}\sum_{k=j+1}^d c_k \overline{\gamma_{k-j}(n)}z^j 
=
\int_{\mathbb T}\overline{e_n(\xi)}\sum_{j=0}^{d-1}\sum_{k=j+1}^d c_k z^j\xi^{\,k-j-1}\operatorname{d}\!\mu(\xi)\\
&=
\int_{\mathbb T}\overline{e_n(\xi)}\sum_{k=1}^d c_k\xi^{k-1}\frac{1-(\overline{\xi}z)^{k}}{1-\overline{\xi}z} \operatorname{d}\!\mu(\xi)=\int_{\mathbb T}\overline{e_n(\xi)}\frac{f(\xi)-f(z)}{\xi-z}\operatorname{d}\!\mu(\xi).
\end{align*}
In conclusion, we found that
\begin{equation*}
    T_{\overline{F_n}}h(z) = \int_{\mathbb{T}} Q_\xi f(z) \overline{e_n(\xi)}\operatorname{d}\!\mu(\xi), \qquad z\in\mathbb{D}.
\end{equation*}
Notice that, since $f$ is a polynomial, \begin{equation}\label{eq:Qzxi}
Q_\xi f(z)=Q_zf(\xi), \qquad z\in\overline{\mathbb{D}},\,\,\xi\in\mathbb{T}.    
\end{equation}
In particular,
\[
|T_{\overline{F_n}}h(z)|^2=|\langle Q_z f, e_n\rangle_{L^2(\mu)}|^2, \qquad z\in\overline{\mathbb{D}}.
\]
It follows by Parseval's identity that
\begin{align*}
\sum_{n=1}^{N}\|T_{\overline{F_n}}h\|_{H^2}^2
&=\int_{\mathbb{T}}\sum_{n=1}^{N}|T_{\overline{F_n}}h(\zeta)|^2\,\operatorname{d}\!m(\zeta)\\
&=
\int_{\mathbb{T}}\|Q_\zeta f\|_{L^2(\mu)}^2\,\operatorname{d}\!m(\zeta)
=
\int_{\mathbb{T}}\|Q_\xi f\|_{H^2}^2\,\operatorname{d}\!\mu(\xi),
\end{align*}
where in the last identity we used \eqref{eq:Qzxi}.
Thus, for every polynomial $h$,
\begin{align*}
\|T_{\overline{a}} h\|_{\mathcal{D}(\mu)}^2& = \|T_{\overline{a}} h\|_{H^2}^2 + \int_{\mathbb{T}} \big\|Q_\xi T_{\overline{a}} h\big\|_{H^2}^2 \operatorname{d}\!\mu(\xi)\\
& = \|T_{\overline{F_0}}h\|_{H^2}^2 + \sum_{n=1}^N \|T_{\overline{F_n}} h\|_{H^2}^2 \asymp \|h\|_{H^2}^2.
\end{align*}
By the density of polynomials in $H^2$, $T_{\overline{a}}$ extends to a bounded and bounded-below operator $H^2\to \mathcal{D}(\mu)$. In particular, the range of $T_{\overline{a}}$ is closed in $\mathcal{D}(\mu)$. Since $a$ is an outer function, the range of $T_{\overline{a}}$ contains all the polynomials. Since polynomials are dense in $\mathcal{D}(\mu)$ (\cite[Theorem 1]{Ri88}), we conclude that $\mathcal{M}(\overline{a}) = \mathcal{D}(\mu)$.

\end{proof}

\section{Further results}\label{S4}

In this section we collect some corollaries of the main theorems, which will be used in the examples of the next sections. 

\begin{lemma}\label{hbeqaulsmbara}
Let $b\in H^\infty_1$ and let $a$ be its Pythagorean mate. If $a+b=1$ or $b^{-1}\in H^\infty$,
then $\mathcal{H}(b)=\mathcal M(\overline a)$ with equivalent norms.
\end{lemma}

\begin{proof}
It is known that $\mathcal M(\overline a)$ is contractively contained in
$\mathcal{H}(b)$~\cite[Lemma 3]{Sa86}. Conversely, $f\in\mathcal{H}(b)$ if and only if
$f\in H^2$ and $T_{\overline b}f=T_{\overline a}f^+$ for some $f^+\in H^2$. Let
$f\in\mathcal{H}(b)$.
\begin{itemize}
    \item If $a+b=1$, then
    $f=T_{\overline a}f+T_{\overline b}f=T_{\overline a}(f+f^+)\in\mathcal M(\overline a)$.
    \item If $b^{-1}\in H^\infty$, then $T_{\overline b}$ is invertible with inverse
    $T_{\overline{1/b}}$, and co-analytic Toeplitz operators commute, so
    $f=T_{\overline{1/b}}T_{\overline b}f=T_{\overline{1/b}}T_{\overline a}f^+
    =T_{\overline a}T_{\overline{1/b}}f^+\in\mathcal M(\overline a)$.
\end{itemize}
Hence $\mathcal{H}(b)=\mathcal M(\overline a)$ as sets. The inclusion
$\mathcal M(\overline a)\hookrightarrow\mathcal{H}(b)$ is then a bounded bijection between
Hilbert spaces, so its inverse is bounded by the open mapping theorem, and the norms are
equivalent.
\end{proof}

Both hypotheses of Lemma \ref{hbeqaulsmbara} imply the Corona condition:
if $a+b=1$ then $|a|+|b|\ge|a+b|=1$, and if $b^{-1}\in H^\infty$ then
$|b|\ge1/\|b^{-1}\|_\infty$. The proof above is a direct argument that does not use this
characterization.

\begin{corollary}\label{maincordmuqhbiffc2}
Let $\mu$ be a finite positive Borel measure on $\mathbb T$. The following are equivalent.
\begin{enumerate}
    \item[(i)] There exists $b$ in $H^\infty_1$ such that $\mathcal D(\mu)=\mathcal{H}(b)$.
    \item[(ii)] There exists $\beta\in H^\infty_1$ with $1/\beta\in H^\infty$ such that
    $\mathcal D(\mu)=\mathcal H(\beta)$.
    \item[(iii)] There exists an outer function $A\in H^\infty_1$ such that
    \[
        |A(z)|^2\bigl(1+V_\mu(z)\bigr)\asymp1,\qquad z\in\mathbb D .
    \]
\end{enumerate}
Moreover, if (iii) holds, then (ii) holds with $\beta$ the outer function with
$\beta(0)>0$ and $|\beta|^2=1-|a|^2$ on $\mathbb T$, where
$a=A/2$, and $\mathcal D(\mu)=\mathcal M(\overline A)=\mathcal H(\beta)$.
\end{corollary}

\begin{proof}
That (ii)$\Rightarrow$(i) is trivial.
(i)$\Rightarrow$(iii) follows from Theorem \ref{T:main2}, taking for $A$ the mate $a$ of $b$. We prove that (iii)$\Rightarrow$(ii). 

Multiplying $A$ by a unimodular constant, we may assume $A(0)>0$. Set
\[
    a(z)=\frac{A(z)}{2},\qquad  \beta(z)=\exp\Bigl\{\frac12\int_{\mathbb T}\frac{\xi+z}{\xi-z}\,\log\bigl(1-|a(\xi)|^2\bigr)\operatorname{d}\!m(\xi)\Bigr\}, \qquad z\in\mathbb{D}.
\]
Then, $\beta$ is outer and $|\beta|^2=1-|a|^2\ge\frac34$ on $\mathbb T$, thus $\beta^{-1}\in H^\infty$ and $(\beta,a)$ is a Pythagorean pair. Since $a=A/2$, by Theorem \ref{T:main2} we have that 
\[
\mathcal{D}(\mu)= \mathcal{H}(\beta)=\mathcal{M}(\overline{a})=\mathcal{M}(\overline{A}). \qedhere
\]
\end{proof}

An advantage of this corollary is that, as soon as $\mathcal D(\mu)$ coincides with some
$\mathcal{H}(b)$, we obtain a canonical symbol $\beta$ that does the job: it is explicit in
terms of $A$ and bounded away from zero.

Finally, we can specialize Theorem \ref{T:main2} to symbols $b=(1+u)/2$ with $u$ inner, the setting of~\cite{BDHM26}.

\begin{corollary}\label{cordmuhohfinniffc}
Let $\mu$ be a finite positive Borel measure on $\mathbb T$ and let $u$ be an inner function.
Then $\mathcal D(\mu)=\mathcal H\bigl((1+u)/2\bigr)$ if and only if
\begin{equation}\label{dmuhohfinniffc}
    |1-u(z)|^2\bigl(1+V_\mu(z)\bigr)\asymp1,\qquad z\in\mathbb D .
\end{equation}
\end{corollary}


\section{$d$-Regular measures}\label{S5}

In this section we give examples of measures $\mu$ with \emph{uncountable} support for which $\mathcal{D}(\mu)$ coincides with a de Branges--Rovnyak space. The examples are produced by the following quantitative condition on $\mu$, which allows us to write down explicitly an outer function $A\in H^\infty$ such that $|A|^2(1+V_\mu)\asymp 1$. We use the standard notation $B(\eta,r):=\{z\in\C:|z-\eta|<r\}$. Also, given two functions $f,g$, by $f\lesssim g$ we mean that $|f(z)|\leq C|g(z)|$ uniformly in $z$ in a specific set, that will be clear from the context. In particular, $f\asymp g$ if and only if $f\lesssim g$ and $g\lesssim f$. Finally, $f_+=\max\{f,0\}.$

\begin{definition}\label{defregular}
A finite positive Borel measure $\mu\ne0$ on $\T$ with support $E$ is \emph{regular} (of dimension $d$) if there are constants $0<c_0<C_0<\infty$ and $0<d<1$ such that
\begin{equation}\label{promeaureg}
    c_0\,r^d\le \mu(B(\eta,r))\le C_0\,r^d,\qquad \eta\in E,\ 0<r\le 1 .
\end{equation}
\end{definition}

This is the usual notion of a $d$-regular (Ahlfors--David regular) measure~\cite{DS93}. We believe that some parts in this section are standard and well-known to experts of geometric measure theory. However, for the sake of completeness, we will include proofs.

Two immediate consequences can be derived from the definition of regularity:
\begin{enumerate}
    \item The upper bound in
\eqref{promeaureg} holds for every radius: if $r>1$, then $r^d>1$, so
$\mu(B(\eta,r))\le\mu(\T)\le\mu(\T)\,r^d$, and therefore, with $C_1=\max\{C_0,\mu(\T)\}$,
\begin{equation}\label{upperallr}
    \mu(B(\eta,r))\le C_1\,r^d,\qquad\eta\in E,\ r>0 .
\end{equation}
\item The measure $\mu$ has no atoms: if $\eta\in E$, then $\mu(\{\eta\})\le\mu(B(\eta,r))\le C_0r^d$
for every $0<r\le1$, and the right-hand side tends to $0$ as $r\to0$; if $\eta\notin E$, then
$\mu(\{\eta\})\le\mu(\T\setminus E)=0$, because $E$ is the support of $\mu$.
\end{enumerate}

We will also show that a $d$-regular measure $\mu$ with $0<d<1$ is necessarily singular with respect to the Lebesgue measure. We will apply the following theorem to the canonical measure on the middle-third Cantor set, see Proposition \ref{cantorexnam}.

\begin{theorem}\label{mainthmmuprnoatom}
Let $\mu$ be a non-zero finite positive Borel measure on $\T$, with support $E$, satisfying \eqref{promeaureg} for some $0<d<1$. Define
\begin{equation*}
    A(z)=\exp\left\{\Bigl(1-\frac d2\Bigr)\int_{\T}\frac{\xi+z}{\xi-z}\,\log\frac{\operatorname{dist}(\xi,\operatorname{supp}(\mu))}{2}\operatorname{d}\!m(\xi)\right\},\qquad z\in\mathbb D .
\end{equation*}
Then $A$ is outer, $A(0)>0$, $\|A\|_\infty\le 1$, and
\[
    |A(z)|^2\bigl(1+V_\mu(z)\bigr)\asymp 1,\qquad z\in\mathbb D .
\]
In particular, there exists $b\in H^\infty_1$ such that
\begin{equation*}
    \mathcal{D}(\mu)=\mathcal{H}(b).
\end{equation*}
\end{theorem}

We fix the following notation for the rest of this section: $\mu$ is a (fixed) non-zero positive finite Borel measure on $\mathbb{T}$ and $E$ its support, a non-empty compact set. We assume that $\mu$ is regular, and we fix the constant $0<d<1$ in \eqref{defregular}.

In order to prove Theorem \ref{mainthmmuprnoatom}, we need two preliminary lemmas.
\begin{lemma}\label{L:estimatevmuregular}
Consider a $d$-regular measure $\mu$ with support $E$. Then,
\begin{equation*}
    1+V_\mu(z)\asymp\operatorname{dist}(z,E)^{d-2},\qquad z\in\mathbb D .
\end{equation*}
\end{lemma}

\begin{proof}
Fix $z\in\mathbb D$, and set $\delta:=\operatorname{dist}(z,E)$. Notice that $\delta < \operatorname{diam}(\overline{\mathbb{D}})=2$. Since $E$ is compact, we can pick $\eta\in E$ such that $|z-\eta|=\delta$.

We begin with the lower bound.  If $\delta>1$, then $\delta^{d-2}<1\le 1+V_\mu(z)$. If $0<\delta\le1$, for $\xi\in B(\eta,\delta)$ we have $|\xi-z|\le|\xi-\eta|+|\eta-z|<2\delta$, whence
\[  V_\mu(z)\ge\int_{B(\eta,\delta)}\frac{\operatorname{d}\!\mu(\xi)}{|\xi-z|^2}\ge\frac{\mu(B(\eta,\delta))}{4\delta^2}\ge\frac{c_0}{4}\,\delta^{d-2}.
\]

We move on to the upper bound. Since $\delta<2$, we have $1\le 2^{2-d}\delta^{d-2}$, so it is enough to show that $V_\mu(z)\lesssim\delta^{d-2}$. For $k\geq 0$, consider the set
\[
    \Gamma_k=\{\xi\in E:\ 2^k\delta\le|\xi-z|<2^{k+1}\delta\}.
\]
Since $\delta=|z-\eta|=\inf\{|z-\xi|:\xi\in E\},$ we have that the (disjoint) sets $\{\Gamma_k\}_{k\geq 0}$  cover $E$. Clearly, this is a finite cover, whose number of sets depends on how close $z$ is to $E$. In any case,
\begin{equation} \label{eq:estimatevmuregular1}
V_\mu(z)=\sum_{k\ge0}\int_{\Gamma_k}\frac{\operatorname{d}\!\mu(\xi)}{|\xi-z|^2}\leq \sum_{k\ge0}\frac{\mu(\Gamma_k)}{(2^k\delta)^2}.
\end{equation}
For every $\xi\in\Gamma_k$, by the triangular inequality
\[
|\xi-\eta|\le|\xi-z|+|z-\eta|<2^{k+1}\delta+\delta\le 3\cdot2^k\delta,
\]
that is, $\Gamma_k\subseteq B(\eta,3\cdot2^k\delta)$.
Using \eqref{upperallr} and \eqref{eq:estimatevmuregular1},
\begin{align*}
V_\mu(z)\le\sum_{k\ge0}\frac{\mu\bigl(B(\eta,3\cdot2^k\delta)\bigr)}{(2^k\delta)^2}
\le C_13^d\sum_{k\ge0}\frac{(2^k\delta)^d}{(2^k\delta)^2}=\frac{C_13^d}{1-2^{d-2}}\,\delta^{d-2},
\end{align*}
completing the proof.
\end{proof}

Before stating the second lemma, we introduce the following notation. We write
\begin{align}
 \label{eqd}    \rho(w)&:=\operatorname{dist}(w,\operatorname{supp}(\mu)), \qquad w\in\overline{\mathbb D}.
\end{align}
Note that $0\le\rho\le 2$ and $0<1-|z|\le\rho(z)<2, z\in\mathbb{D}$. Also, $P_z(\zeta)=\frac{1-|z|^2}{|z-\zeta|^2},$ for $z\in\mathbb{D}$ and $\zeta\in\mathbb{T}.$

\begin{lemma}\label{logintest}
Consider a $d$-regular measure $\mu$ with support $E$. The following facts hold:
\begin{itemize}
    \item[(i)] The support of $\mu$ has Lebesgue measure $m(E)=0$, and $\log\rho\in L^1(m)$.
    \item[(ii)] There exists a constant $K$ such that
\[
    \log\rho(z)-K\ \le\ \int_{\T}P_z(\xi)\log\rho(\xi)\operatorname{d}\!m(\xi)\ \le\ \log\rho(z),\qquad z\in\mathbb D .
\]
\end{itemize}
\end{lemma}

\begin{proof}
For $s>0$ let 
\begin{equation}\label{eq:defUslogintest}
    U_s=\{\xi\in\T:\operatorname{dist}(\xi,E)<s\}.
\end{equation} 
By definition, $E\subseteq U_s$ for every $s>0$. We want to prove the following local estimate for $m(U_s)$: for every $\eta\in E$ and $R>0$,
\begin{equation}\label{loestesb}
    m\bigl(U_s\cap B(\eta,R)\bigr)\lesssim R^d s^{1-d},\qquad  0<s\le\min\{R,1\}.
\end{equation}
Let $\eta_1,\dots,\eta_N$ be a maximal $s$-separated subset of $E\cap B(\eta,R+s)$, that is, $|\eta_i-\eta_j|\ge s$ for $i\ne j$ and no other point of $E\cap B(\eta,R+s)$ can be added to the sequence without violating this condition, i.e.,
\[
\forall \xi\in E\cap B(\eta,R+s) \qquad \exists j\in\{1,\ldots,N\} \qquad \text{s.t.}\quad|\xi-\eta_j|<s.
\]
Such a set is non-empty and finite, since $E\cap B(\eta,R+s)$ is not empty (it contains $\eta$) and bounded. We give an upper bound for the cardinality $N$. For every $\xi$ in $\mathbb{T}$,
\[
|\xi-\eta|\leq |\xi-\eta_j|+|\eta_j-\eta|\leq |\xi-\eta_j| + R+s, \qquad j=1,\ldots,N.
\]
In particular, the balls $B(\eta_j,s/3)$ are pairwise disjoint and they are all contained in $B(\eta,3R)$, since $s\le R$. Using the lower bound in \eqref{promeaureg} (note $s/3<1$) and \eqref{upperallr},
\[
    N\,c_0\,3^{-d}s^d\le\sum_{j=1}^N\mu\bigl(B(\eta_j,s/3)\bigr)\le\mu\bigl(B(\eta,3R)\bigr)\le C_13^dR^d .
\]
Hence $N\lesssim(R/s)^d$. We proceed to prove \eqref{loestesb}.

Let $\xi\in U_s\cap B(\eta,R)$. There exists $e\in E$ with $|\xi-e|<s$, so $e\in E\cap B(\eta,R+s)$, and by maximality $|e-\eta_j|<s$ for some $j$. Thus $|\xi-\eta_j|<2s$, that is, $U_s\cap B(\eta,R)\subseteq\bigcup_{j=1}^NB(\eta_j,2s)$.
Since $m(B(\eta_j,2s)\cap\T)\lesssim s$,
\[
    m\bigl(U_s\cap B(\eta,R)\bigr)\lesssim N s\lesssim R^ds^{1-d},
\]
which is \eqref{loestesb}.

\textit{Proof of}\,(i). Fix $\eta\in E$. Since $\T\subseteq B(\eta,3)$, taking $R=3$ in \eqref{loestesb} gives
\begin{equation}\label{reqaltest}
    m(U_s)\lesssim s^{1-d},\qquad 0<s\le1 .
\end{equation}
As $E\subseteq U_s$ for every $s$, letting $s\to0$ gives $m(E)=0$.
Next, since 
\[
\{(-\log\rho)_+>t\}=\{\rho<e^{-t}\}=U_{e^{-t}}, \qquad t>0,
\]
the substitution $s=e^{-t}$ and \eqref{reqaltest} give for $f:=(-\log\rho)_+$
\[
    \int_{\T}f\operatorname{d}\!m=\int_0^\infty\!\! m(\{f>t\})\operatorname{d}\!t=\int_0^\infty\!\! m(U_{e^{-t}})\operatorname{d}\!t=\int_0^1m(U_s)\,\frac{\operatorname{d}\!s}{s}\lesssim\int_0^1s^{-d}\operatorname{d}\!s<\infty .
\]
Since also $\rho\le2$, we get $(\log\rho)_+\le\log2$, hence $\log\rho\in L^1(m)$.

\textit{Proof of}\,(ii). The upper bound is easy. For $z\in\mathbb{D}$, let $\eta\in E$ be such that $|z-\eta|=\rho(z)$. Then, for every $\xi\in\mathbb{T}$, $\rho(\xi)\le|\xi-\eta|$. Moreover $|\xi-\eta|=|1-\bar\eta\xi|$ on $\T$, and $1-\bar\eta z$ is outer, hence
\[
\int_{\T}P_z\log\rho\operatorname{d}\!m\le\int_{\T}P_z(\xi)\log|\xi-\eta|\operatorname{d}\!m(\xi)=\log|1-\bar\eta z|=\log|z-\eta|=\log\rho(z).
\]

We move on to the lower bound. Fix $z\in\mathbb D$, put $\delta=\rho(z)=\operatorname{dist}(z,E)$ and pick $\eta\in E$ with $|z-\eta|=\delta$. Since  $\int_{\T}P_z\operatorname{d}\!m=1$,
\[
\log\delta-\int_{\T}P_z(\xi)\log\rho(\xi)\operatorname{d}\!m(\xi)\leq \int_{\T}P_z(\xi)\Bigl(\log\frac{\delta}{\rho(\xi)}\Bigr)_+\operatorname{d}\!m(\xi).
\]
Then, it suffices to prove that
\[
\int_{\T}P_z(\xi)\Bigl(\log\frac{\delta}{\rho(\xi)}\Bigr)_+\operatorname{d}\!m(\xi) \leq K<\infty,
\]
where $K$ does not depend on $z$. The argument is similar to that in the proof of (i). By the substitution $s=\delta e^{-t}$,
\begin{equation}\label{eq:estimatePzlog1}
\int_{\T}P_z\Bigl(\log\frac\delta\rho\Bigr)_+\operatorname{d}\!m
=\int_0^\infty\!\int_{\{\rho<\delta e^{-t}\}}P_z\operatorname{d}\!m\operatorname{d}\!t
=\int_0^\delta\Bigl(\int_{U_s}P_z\operatorname{d}\!m\Bigr)\frac{\operatorname{d}\!s}{s},
\end{equation}
where $U_s$ is as in \eqref{eq:defUslogintest}. We first show that
\begin{equation}\label{poissonUs}
\int_{U_s}P_z\operatorname{d}\!m\lesssim\Bigl(\frac s\delta\Bigr)^{1-d},\qquad 0<s<\delta/2,
\end{equation}

If $\xi\in U_s$, there is $e\in E$ with $|\xi-e|<s$, hence $|\xi-z|\ge|e-z|-s\ge\delta-s>\delta/2$. Therefore $U_s$ is covered by the (disjoint) sets
\[
    S_k=\{\xi\in\T:\ 2^{k-1}\delta\le|\xi-z|<2^k\delta\},\qquad k\ge0 .
\]
In particular,
\[
\int_{U_s}P_z\operatorname{d}\!m=\sum_{k\ge0}\int_{U_s \cap S_k} P_z\operatorname{d}\!m\le\sum_{k\ge0}\frac{8(1-|z|)}{(2^k\delta)^2}\,m(U_s\cap S_k).
\]
By the triangular inequality, one can prove that $S_k\subseteq B(\eta,2^{k+1}\delta), k\geq 0$. Then, by \eqref{loestesb} applied to $R=2^{k+1}\delta$ 
\[
    m(U_s\cap S_k)\le m\bigl(U_s\cap B(\eta,2^{k+1}\delta)\bigr)\lesssim(2^k\delta)^ds^{1-d}, \qquad 0<s<\min\{1,R\}.
\]
Since $\delta<2$, for every $0<s<\delta/2$ we have that
\begin{align*}
\int_{U_s}P_z\operatorname{d}\!m&\lesssim\sum_{k\ge0}\frac{1-|z|}{(2^k\delta)^2}(2^k\delta)^ds^{1-d}\\
&=\frac{1-|z|}{\delta}\Bigl(\frac s\delta\Bigr)^{1-d}\sum_{k\ge0}(2^{d-2})^k\lesssim \Bigl(\frac s\delta\Bigr)^{1-d},
\end{align*}
which proves \eqref{poissonUs}. We have also used that $1-|z|\leq \delta=\rho(z)$.

We now conclude the proof. Let $C_P>0$ be a constant such that
\[
    \int_{U_s}P_z\operatorname{d}\!m\le C_P(s/\delta)^{1-d},\qquad 0<s<\delta/2,
\]
as established in \eqref{poissonUs}. Then, by \eqref{eq:estimatePzlog1},
\[
\begin{aligned}
\int_{\T}P_z\Bigl(\log\frac\delta\rho\Bigr)_+\operatorname{d}\!m
&\le C_P\int_0^{\delta/2}\Bigl(\frac s\delta\Bigr)^{1-d}\frac{\operatorname{d}\!s}{s}+\int_{\delta/2}^{\delta}\frac{\operatorname{d}\!s}{s}
=C_P\frac{2^{d-1}}{1-d}+\log2=:K,
\end{aligned}
\]
where we used \eqref{poissonUs} on $(0,\delta/2)$ and the trivial bound $\int_{U_s}P_z\operatorname{d}\!m\le1$ on $[\delta/2,\delta]$. The constant $K$ does not depend on $z$. 
\end{proof}

We can now prove the main theorem of the section.

\begin{proof}[Proof of Theorem \ref{mainthmmuprnoatom}]
By Lemma \ref{logintest}, $\log(\rho/2)\in L^1$ and $\log(\rho/2)\le0$. Thus, $A$ is the outer function with boundary modulus $(\rho/2)^{1-d/2}$. In particular, since $0\le\rho\le2$, $\|A\|_\infty\le1$. Notice also that $A(0)=\exp\{(1-\tfrac d2)\int_\T\log(\rho/2)\operatorname{d}\!m\}>0$.

For $z\in\mathbb D$, since $A$ is outer
\[
    |A(z)|^2=\exp\Bigl\{(2-d)\Bigl(\int_{\T}P_z\log\rho\operatorname{d}\!m-\log2\Bigr)\Bigr\}\asymp\rho(z)^{2-d},
\]
where we used the two-sided bound of Lemma \ref{logintest}. Together with Lemma \ref{L:estimatevmuregular},
\[
    |A(z)|^2\bigl(1+V_\mu(z)\bigr)\asymp\rho(z)^{2-d}\,\rho(z)^{d-2}=1.
\]
The conclusion now follows from Corollary \ref{maincordmuqhbiffc2}.
\end{proof}

\subsection{An explicit example: the Cantor measure} 

We now give a concrete example. We recall the construction of the Cantor set $C$, the Cantor function $f_C$ and the Cantor measure $\omega_C$. For more details, we refer to \cite{Cohn2013, Folland1999}.

Let $K_0$ be the interval $[0,1]$. In the step $1$, form $K_1$ by removing from $K_0$
the interval $(1/3,2/3)$. Thus,
\[
K_1=\left[0,\frac13\right]\cup\left[\frac23,1\right]\!.
\]
Continue this procedure: at each step $n$ we form $K_n$ by removing from $K_{n-1}$ the open middle third of each of the intervals making up $K_{n-1}$. At the step $n$, we remove $2^{n-1}$ intervals from $K_{n-1}$, and $K_n$ is the union of $2^n$ disjoint closed intervals, namely
\[
K_n = \bigcup_{j=1}^{2^n} I_{n,j},
\]
each of length $3^{-n}$. The Cantor set $C$ is the set of points that remain after infinite iterations, i.e.,
\[
C=\bigcap_{n=0}^{\infty}K_n.
\]
It is known $C$ is an uncountable compact set with zero Lebesgue measure. We now define the Cantor function $f_C\colon[0,1]\to[0,1]$. 
For each $x\in(1/3,2/3)$, let $f_C(x)=1/2$.
Next, define $f_C$ on the intervals removed from $K_1$ in the
construction of $K_2$ by setting
\[
f_C(x)=
\begin{cases}
\frac14, & x\in(1/9,2/9),\\[4pt]
\frac34, & x\in(7/9,8/9).
\end{cases}
\]
In general, at each step $n$ we remove $2^{n-1}$ intervals from $K_{n-1}$. Ordering these intervals in increasing order, we define $f_C$ to be the constant $(2j-1)/2^n$ on the $j$-th removed interval, $j=1,\ldots, 2^{n-1}$. In this way, $f_C$ is defined on $[0,1]\setminus C$,
is non-decreasing, and takes values in $[0,1]$.
We extend $f_C$ to all of $[0,1]$ by setting $f_C(0)=0$ and
\[
f_C(x)=\sup\bigl\{
f_C(t):t\in[0,1]\setminus C,\ t<x
\bigr\},
\qquad x\in C\setminus\{0\}.
\]
The resulting function $f_C$ is continuous and non-decreasing, with $f_C(0)=0$ and $f_C(1)=1$. We can extend it to $\mathbb{R}$ by continuity by setting $f_C(x)=0, x<0,$ and $f_C(x)=1,x>1.$ Then, there exists a Lebesgue--Stieltjes measure $\omega_C$ associated to $f_C$, that satisfies
\[
\omega_C((a,b])=f_C(b)-f_C(a), \qquad a\leq b.
\]
In particular, $\omega_C$ is a finite positive Borel measure that is supported on $C$ (then, in particular it is singular with respect to the Lebesgue measure), and has no atoms. Moreover, at every step $n$, on every interval $I_{n,j}$ that forms $K_n$, 
\[
\omega_C(I_{n,j})=2^{-n}.
\]
Set $d=\log2/\log3$. We show that $\omega_C$ is $d$-regular. Notice that $d$ appears in the relation between length and Cantor measure of $I_{n,j}$:
\[
|I_{n,j}|^d = (3^{-n})^d = 2^{-n}=\omega_C(I_{n,j}).
\]
We prove that 
\begin{equation*}
    \omega_C\big((x-r,x+r)\big)\asymp r^d,\qquad x\in C,\ 0<r\le 1.
\end{equation*}
Fix $x\in C$ and $0<r\le1$, and choose $k\ge0$ with $3^{-(k+1)}<r\le3^{-k}$. Then
\[
    \tfrac12\,2^{-k}\ \ge\ \tfrac12\, r^d\qquad\text{and}\qquad r^d>3^{-(k+1)d}=2^{-(k+1)} .
\]
Since $x\in C,$ in particular, $x$ belongs to one of the intervals $I_{k+1,j}$, for some $j$. Since $I_{k+1,j}$ has length $3^{-(k+1)}<r$, $I_{k+1,j}\subseteq(x-r,x+r)$. Therefore
\[
    \omega_C\bigl((x-r,x+r)\bigr)\ge\omega_C(I_{k+1,j})=2^{-(k+1)}\ge\tfrac12\,r^d .
\]
We move on to the upper bound. The interval $(x-r,x+r)$ has length at most $2\cdot3^{-k}$ and the gaps between distinct intervals $I_{k,j}$ are at least $3^{-k}$, so it meets at most three of them. Thus
\[
    \omega_C\bigl((x-r,x+r)\bigr)\le3\cdot2^{-k}<6\,r^d .
\]

We now transfer the Cantor measure on the unit circle. In general, if $\omega$ is a positive measure on $[0,1]$ and $\phi(t)=e^{it}$, let $\mu=\phi_*\omega$ be the push-forward measure,
\[
    \mu(F)=\omega\bigl(\{t\in[0,1]:\phi(t)\in F\}\bigr),\qquad\text{ for } F\subseteq\T\ \text{Borel}.
\]

\begin{proposition}\label{cantorexnam} Let $\omega_C$ be the Cantor measure on $[0,1]$, and consider the push-forward measure $\mu=\phi_*\omega_C$ on $\mathbb{T}$. Then, $\mu$ is a non-atomic probability measure with uncountable support $E=\{e^{it}:t\in C\}$, and there exists a function $b\in H^\infty_1$ such that
\[
    \mathcal{D}(\mu)=\mathcal{H}(b).
\]
\end{proposition}

\begin{proof}
We show that $\mu$ is $d$-regular with $d=\log2/\log3$. For $s,t\in[0,1]$ we have $|e^{is}-e^{it}|=2\sin\frac{|s-t|}2$, hence
\begin{equation}\label{eq:cantorexnam1}
    \tfrac2\pi|s-t|\le|e^{is}-e^{it}|\le|s-t| .
\end{equation}
Then, since the Cantor measure $\omega_C$ is $d$-regular, so it is $\mu$.
Moreover, $\mu$ inherits from $\omega_C$ all the properties in the statement. The conclusion follows by applying Theorem \ref{mainthmmuprnoatom}.
\end{proof}

We remark that the structure of the invariant subspaces of $\mathcal{D}(\mu)$ is described in \cite[Corollary 1.1]{EEL22}.

\section{Measures concentrated on a countable set}\label{S6}

In this section we consider purely atomic measures $\mu$ and give a condition on the atoms, \eqref{assumpsummab} below, under which $\mathcal{D}(\mu)$ coincides with a space $\mathcal{H}(b)$, with $b$ written explicitly in terms of the atoms.

We introduce the setting that we will consider in the rest of the section. Let $(\zeta_n)_{n\geq 1}$ be distinct points in $\mathbb{T}$ and $(s_n)_{n\geq 1}$ a summable sequence of positive numbers. Then, the measures
\[
\sigma=\sum_{n\geq 1}s_n \delta_{\zeta_n}, \qquad \mu=\sum_{n\geq 1}s_n^2 \delta_{\zeta_n}
\]
are finite, and their support is $E=\overline{\{\zeta_n:n\ge1\}}$. Notice that $E$ can be uncountable (see Proposition \ref{uncounsptcm}). We consider the Herglotz transform of $\sigma$, that is,
\begin{equation*}
    F(z):=\int_{\mathbb{T}}\frac{\zeta+z}{\zeta-z}\operatorname{d}\!\sigma(\zeta)=\sum_ns_n\frac{\zeta_n+z}{\zeta_n-z},\qquad  z\in\mathbb{D}.
\end{equation*}
Then, the function
\begin{align*}
     u(z)&=\frac{F(z)-1}{F(z)+1}, \qquad z\in\mathbb{D},
\end{align*}
is an inner function, and its Clark measure at $1$ coincides with $\sigma$. See for example \cite[Remark 9.14]{CMR06}. Finally, set
\begin{equation*}
     a(z)=\frac{1-u(z)}2=\frac1{1+F(z)}, \qquad b(z)=\frac{1+u(z)}2=\frac{F(z)}{1+F(z)}, \qquad z\in\mathbb{D}.
\end{equation*}
Notice that $a(0)>0.$ In particular, $(b,a)$ is a Pythagorean pair that satisfies the Corona condition.

\begin{theorem}\label{mainthmcountablespt}
Let
\[
    \mu=\sum_{n\ge1}s_n^2\,\delta_{\zeta_n},\quad s_n>0,\quad\zeta_n\in\T\ \text{distinct},\quad\sum_ns_n^2<\infty,
\]
and let $E=\operatorname{supp}\mu=\overline{\{\zeta_n:n\ge1\}}$. Suppose that
\begin{equation}\label{assumpsummab}
    M:=\sup_n\sum_{j\ne n}\frac{s_j}{|\zeta_j-\zeta_n|}<\infty.
\end{equation}
Then $\sum_ns_n<\infty$. In particular, the functions $F,u,b,a$ introduced above are well-defined, and we have that
\[
    \mathcal{D}(\mu)=\mathcal{H}(b)=\mathcal M(\bar a).
\]
\end{theorem}

Roughly, \eqref{assumpsummab} says that the mass of each atom is small compared with its distance from the other atoms, in a summable sense.

Before proving the theorem, we need the following lemma. 
\begin{lemma}\label{lemconseq}
With the same notation as above, assume \eqref{assumpsummab}. Then:
\begin{enumerate}
    \item[(a)] $S:=\sum_ns_n<\infty$.
    \item[(b)] $s_j\le M|\zeta_j-\zeta_n|$ for $j\ne n$. Every $\zeta_n$ is an isolated point of $E$, and
    \[
        \rho_n:=\operatorname{dist}(\zeta_n,E\setminus\{\zeta_n\})=\inf_{j\ne n}|\zeta_n-\zeta_j|>0,\qquad s_n\le M\rho_n,\qquad \rho_n\le2 .
    \]
    \item[(c)] $\displaystyle\sup_n\sum_{j\ne n}\frac{s_j^2}{|\zeta_j-\zeta_n|^2}\le M^2$.
    \item[(d)] For every $\eta\in E\setminus\{\zeta_n:n\ge1\}$,
    \begin{equation}\label{accumulestimatee}
        \sum_j\frac{s_j}{|\zeta_j-\eta|}\le M .
    \end{equation}
\end{enumerate}
\end{lemma}

\begin{proof}
We prove the four items separately.
\begin{itemize}
    \item[(a)] Since $|\zeta_j-\zeta_n|\le2$, for any fixed $n$ we get $\sum_{j\ne n}s_j\le2\sum_{j\ne n}\frac{s_j}{|\zeta_j-\zeta_n|}\le2M$, hence $S\le2M+s_n<\infty$.
    \item[(b)] Each term of the sum in \eqref{assumpsummab} is at most $M$, i.e.\ $s_n\le M|\zeta_j-\zeta_n|$ for all $n\ne j$. In particular, we obtain that $\inf_{j\neq n}|\zeta_n-\zeta_j|\ge s_n/M$. Now $E=\{\zeta_n\}\cup\overline{\{\zeta_j:j\ne n\}}$, and every point of $\overline{\{\zeta_j:j\ne n\}}$ is at distance at least $s_n/M$ from $\zeta_n$ (as a limit of points with this property). Hence $B(\zeta_n,s_n/M)\cap E=\{\zeta_n\}$, so $\zeta_n$ is isolated, $E\setminus\{\zeta_n\}=\overline{\{\zeta_j:j\ne n\}}$ and the formula for $\rho_n$ follows. The bounds $s_n\le M\rho_n$ and $\rho_n\le2$ are now clear.
    \item[(c)] By (b), each term satisfies $\frac{s_j^2}{|\zeta_j-\zeta_n|^2}\le M\frac{s_j}{|\zeta_j-\zeta_n|}$, hence the conclusion.
    \item[(d)] Since $\eta\in E$ is not an atom, there are distinct atoms $\zeta_{n_k}$ converging to $\eta$. Fix a finite set $J$ of indices. For $k$ large, $n_k\notin J$, and by \eqref{assumpsummab},
\[
\sum_{j\in J}\frac{s_j}{|\zeta_j-\zeta_{n_k}|}\le M.
\]
Letting $k\to\infty$ gives $\sum_{j\in J}\frac{s_j}{|\zeta_j-\eta|}\le M$. We conclude taking the supremum over $J$. \qedhere
\end{itemize}
\end{proof}

We are in a position to prove the main theorem of the section.

\begin{proof}[Proof of Theorem \ref{mainthmcountablespt}]
By Lemma \ref{lemconseq}(a), $S=\sum_{n\geq 1}s_n<\infty$, so that the Herglotz transform of $\sigma=\sum_{n\geq 1}s_n\delta_{\zeta_n}$, i.e.,
\[
 F(z)=\sum_{n\geq 1}s_n\frac{\zeta_n+z}{\zeta_n-z},\qquad  z\in\mathbb{D},
\]
produces an analytic function on $\mathbb{D}$, and $u=(F-1)/(F+1)$ is an inner function, whose Clark measure $\sigma_1$ coincides with $\sigma.$ Moreover, the functions 
\begin{equation*}
     a(z)=\frac{1-u(z)}2=\frac1{1+F(z)}, \qquad b(z)=\frac{1+u(z)}2=\frac{F(z)}{1+F(z)}, \qquad z\in\mathbb{D},
\end{equation*}
are both outer and positive at $0$, and satisfy $a+b=1$. In particular, $(b,a)$ is a Pythagorean pair that satisfies the Corona condition. Since 
\[
|a|^2(1+V_\mu)=\frac{1+V_\mu}{|1+F|^2},
\]
in order to apply Corollary \ref{cordmuhohfinniffc} we have to show that $|1+F|^2\asymp1+V_\mu$ on $\mathbb D$. We will introduce a family of disjoint disks, and consider the different cases, whether $z$ lies in such a disk or not.

Fix $0<\varepsilon<\tfrac14$, and consider the disks 
\begin{equation*}  
U_n=\{z\in\C:|z-\zeta_n|<\varepsilon\rho_n\}, \qquad n\geq 1,
\end{equation*}
with $ \rho_n=\operatorname{dist}(\zeta_n,E\setminus\{\zeta_n\})=\inf_{j\ne n}|\zeta_n-\zeta_j|>0$ as in Lemma \ref{lemconseq}(b). 

These disks are pairwise disjoint: for $z\in U_n$ and $j\ne n$, using $\rho_n\le|\zeta_n-\zeta_j|$,
\begin{equation}\label{Unlower}
    |z-\zeta_j|\ge|\zeta_n-\zeta_j|-\varepsilon\rho_n\ge(1-\varepsilon)|\zeta_n-\zeta_j| > \frac34\rho_j.
\end{equation}

\textit{Case 1: $z\in U_n\cap\mathbb{D}$ for some $n\geq 1$.} 
We begin with estimating the potential
\[
V_\mu(z)=\frac{s_n^2}{|z-\zeta_n|^2}+\sum_{j\ne n}\frac{s_j^2}{|z-\zeta_j|^2}.
\]
By \eqref{Unlower} and Lemma \ref{lemconseq}(c),
\[
    \sum_{j\ne n}\frac{s_j^2}{|z-\zeta_j|^2}\le\frac{M^2}{(1-\varepsilon)^2}.
\]

We conclude that
\begin{equation}\label{eq:mainthmcountablespt2}
1+V_\mu(z) \asymp 1+ \frac{s_n^2}{|z-\zeta_n|^2}, \qquad z\in U_n\cap\mathbb{D}.    
\end{equation}

We want analogous estimates for $|1+F|^2$. Write \[
F(z)=s_n\frac{\zeta_n+z}{\zeta_n-z}+\sum_{j\ne n}s_j\frac{\zeta_j+z}{\zeta_j-z},
\]
and denote $F_n$ the second summand above. Since $\frac{\zeta_n+z}{\zeta_n-z}=-\frac{2\zeta_n}{z-\zeta_n}-1$, we get
\begin{equation}\label{eq:mainthmcountablespt1}
    1+F(z)=-\frac{2s_n\zeta_n}{z-\zeta_n}+H_n(z),\qquad H_n=1-s_n+F_n.
\end{equation}
    
We estimate $|H_n|$. First, notice that by \eqref{Unlower} and \eqref{assumpsummab},
\[
\begin{aligned}
    |F_n(z)|&\le\frac2{1-\varepsilon}\sum_{j\ne n}\frac{s_j}{|\zeta_j-\zeta_n|}\le\frac{2M}{1-\varepsilon}<\frac{8M}3.
\end{aligned}
\]
By Lemma \ref{lemconseq}(b), we have that $s_n\le M\rho_n\le2M$. It follows that
\[
|H_n(z)|\le1+2M+\tfrac83M\le 1+5M =: M_1.
\]
Then, by \eqref{eq:mainthmcountablespt1}, we proved that
\begin{equation*}
    \bigg|\frac{2s_n}{|z-\zeta_n|}- |H_n(z)|\bigg| \leq |1+F(z)|\leq \frac{2s_n}{|z-\zeta_n|}+ M_1, \qquad z\in U_n\cap\mathbb D.
\end{equation*}

We obtain the upper estimate 
\[
|1+F(z)|^2\lesssim \frac{s_n^2}{|z-\zeta_n|^2}+ 1.
\]

For the lower estimate, set $x_n:=s_n/|z-\zeta_n|$. We distinguish between two cases:
\begin{itemize}
    \item[(i)] If $x_n\ge M_1$ (note $M_1\ge1$), then $|1+F(z)|\ge x_n\ge1$ and $x_n^2\ge\frac12(1+x_n^2)$. Thus,
    \[
    |1+F(z)|^2\geq \frac12+\frac{s_n^2}{2|z-\zeta_n|^2}.
    \]
    \item[(ii)] Recall that $|1+F(z)|=|a(z)|^{-1}\ge1$ for $z\in\mathbb D$. If $x_n< M_1$, then
\[
|1+F|^2\ge1\ge\frac{1+x_n^2}{1+M_1^2}.
\]
\end{itemize}

We conclude that
\[
    |1+F(z)|^2\asymp1+x_n^2=1+\frac{s_n^2}{|z-\zeta_n|^2},\qquad z\in U_n\cap\mathbb D .
\]
Therefore, by \eqref{eq:mainthmcountablespt2},  $|1+F|^2\asymp1+V_\mu$ on $U_n\cap\mathbb D$.

\textit{Case 2: $z$ outside of all the disks $U_n$.}  Let $z\in\mathbb D\setminus\bigcup_nU_n$ and let $\eta\in E=\operatorname{supp}(\mu)$ be a point nearest to $z$ ($E$ is compact). For every $j\geq 1$, 
\[
|\zeta_j-\eta|\le|\zeta_j-z|+|z-\eta|\le2|\zeta_j-z|,
\]
because $|z-\eta|=\operatorname{dist}(z,E)\le|z-\zeta_j|$. There are two cases.
\begin{itemize}
    \item If $\eta$ is not an atom, then by \eqref{accumulestimatee}, \[
    \sum_{j\geq 1}\frac{s_j}{|z-\zeta_j|}\le2\sum_{j\geq 1}\frac{s_j}{|\zeta_j-\eta|}\le2M.
    \]
    \item If $\eta=\zeta_n$, the same bound gives 
    \[
    \sum_{j\ne n}\frac{s_j}{|z-\zeta_j|}\le2\sum_{j\ne n}\frac{s_j}{|\zeta_j-\zeta_n|}\le2M,
    \]
    while the remaining term satisfies \[
    \frac{s_n}{|z-\zeta_n|}\le\frac{s_n}{\varepsilon\rho_n}\le\frac M\varepsilon,
    \]
    since $z\notin U_n$ and $s_n\leq M\rho_n$, by Lemma \ref{lemconseq}(b).
\end{itemize}
In both cases $\sum_{j\geq 1}\frac{s_j}{|z-\zeta_j|}\le K_0:=2M+\frac M\varepsilon$. Consequently,
\[
    |F(z)|\le\sum_{j\geq 1}\frac{2s_j}{|z-\zeta_j|}\le2K_0,\qquad V_\mu(z)=\sum_{j\geq 1}\frac{s_j^2}{|z-\zeta_j|^2}\le\Bigl(\sum_{j\geq 1}\frac{s_j}{|z-\zeta_j|}\Bigr)^2\le K_0^2 .
\]
So $1\le|1+F|\le1+2K_0$ and $1\le1+V_\mu\le1+K_0^2$, and again $|1+F|^2\asymp1+V_\mu$.

 We found that $|1+F|^2\asymp1+V_\mu$ uniformly on the whole disk $\mathbb D$, with constants depending only on $M$ and $\varepsilon$. Hence
\[
    |a(z)|^2\bigl(1+V_\mu(z)\bigr)=\frac{1+V_\mu(z)}{|1+F(z)|^2}\asymp1,\qquad z\in\mathbb D .
\]
Corollary \ref{cordmuhohfinniffc} now gives $\mathcal{D}(\mu)=\mathcal M(\bar a)=\mathcal{H}(b)$.
\end{proof}

We remark that the setting is very similar to the one considered in Theorem 2.3 of \cite{BDHM26}. However, those results cannot be applied here, because $u$ is not one-component:

\begin{remark}
    Consider $\mu$ as in Theorem \ref{mainthmcountablespt}, and assume that \eqref{assumpsummab} holds. Suppose by contradiction that $u$ is one-component. 
    
    Consider $\Lambda\ne\varnothing$, the set of accumulation points of the atoms of its Clark measure $\sigma$. By \cite[Theorem 1]{Bes15}, every atom has two neighbouring atoms, and its mass is bounded below by a fixed positive multiple of either neighbouring distance. In a component of $\T\setminus\Lambda$ containing an atom, one can therefore choose consecutive atoms $\xi_k=pe^{it_k}$, $k\geq 1$, approaching an endpoint $p\in\Lambda$, with $0<t_{k+1}<t_k<\pi$ and $t_k\to0$; the opposite orientation is treated identically. Writing $w_k=\sigma_1(\{\xi_k\})$, Bessonov's lower estimate gives
\[
    \frac{w_k}{|\xi_k-p|}\gtrsim\frac{t_k-t_{k+1}}{t_k}.
\]
Setting $q_k=1-t_{k+1}/t_k,k\geq 1$, we claim that $\sum_{k\geq 1}q_k=+\infty$. If this is true, then in particular
\[
    \sum_{k\geq 1}\frac{w_k}{|\xi_k-p|}=\infty,
\]
contradicting Lemma \ref{lemconseq}(d).

Assume that $\sum_{k\geq 1}q_k<+\infty$. Then, in particular, $q_k$ and $x_k:=-\log(1-q_k)$ are two sequences of positive numbers such that $\lim_k x_k/q_k=1$. However,
\[
 \sum_{k=1}^m-\log(1-q_k)=\sum_{k=1}^m \big(\log(t_k)-\log(t_{k+1})\big)= \log\frac{t_1}{t_{m+1}}\xrightarrow[m\to\infty]{}+\infty,
\]
which is a contradiction.
\end{remark}

\subsection{An example}
The following result gives a measure satisfying the hypotheses of Theorem \ref{mainthmcountablespt} with uncountable support.

\begin{proposition}\label{uncounsptcm}
Let $C\subset[0,1]$ be the middle-third Cantor set and let $\mathcal G$ be the (countable) family of all open intervals removed in its construction. For $I\in\mathcal G$ write
\[
    \ell_I=|I|,\qquad c_I=\text{the midpoint of }I,\qquad\zeta_I=e^{ic_I}.
\]
Define
\begin{equation*}
    \mu=\sum_{I\in\mathcal G}\ell_I^4\,\delta_{\zeta_I},\qquad
    F(z)=\sum_{I\in\mathcal G}\ell_I^2\,\frac{\zeta_I+z}{\zeta_I-z},\qquad a=\frac1{1+F},\qquad b=\frac F{1+F}.
\end{equation*}
Then $\operatorname{supp}\mu=\{e^{it}:t\in C\}\cup\{\zeta_I:I\in\mathcal G\}$ is uncountable and
\[
    \mathcal{D}(\mu)=\mathcal M(\bar a)=\mathcal{H}(b).
\]
\end{proposition}

\begin{proof}
At step $k\ge1$ of the construction of $C$ there are $2^{k-1}$ removed intervals, of length $3^{-k}$. Hence
\[
    \mu(\T)=\sum_{k\ge1} 3^{-4k} 2^{k-1}=\frac12\sum_{k\ge1}\Bigl(\frac2{81}\Bigr)^k
    <\infty,
    \qquad \sum_{I\in\mathcal G}\ell_I=\sum_{k\ge1}2^{k-1}3^{-k}=1 .
\]
The points $\zeta_I$ are distinct, and $\mu$ has the form required in Theorem \ref{mainthmcountablespt} with $s_I=\ell_I^2$. We claim that
\[
\overline{\{c_I:I\in\mathcal G\}}=C\cup\{c_I:I\in\mathcal G\}.
\]
We show that $C\subseteq \overline{\{c_I:I\in\mathcal G\}}$, establishing one inclusion. Every $t\in C$ lies in a nested sequence of intervals with length $3^{-n}$, and each of them contains the removed middle third, hence a midpoint $c_I$; these points $c_I$ accumulate at $t$.

We show the other inclusion, by proving that $\overline{\{c_I:I\in\mathcal G\}}\setminus \{c_I:I\in\mathcal G\}\subseteq C.
$ The set of removed intervals of length $\ge\varepsilon$ is finite for each $\varepsilon>0$, so if $c_{I_k}\to t$ with distinct $I_k\in\mathcal{G}$, then $\ell_{I_k}\to0$. The endpoints of $I_k$ belong to $C$ and their distance from $c_{I_k}$ is  $\ell_{I_k}/2$, hence $t\in \overline{C}=C$, proving the claim.

The map $\phi(t)=e^{it}$  on the compact set $[0,1]$ is a homeomorphism onto its image, so it commutes with closures. As $\mu$ is purely atomic with all atoms of positive mass,
\[
    \operatorname{supp}\mu=\overline{\{\zeta_I\}}=\{e^{it}:t\in C\}\cup\{\zeta_I:I\in\mathcal G\},
\]
which is uncountable.

Now we prove that $\mu$ satisfies condition \eqref{assumpsummab}. Let $I\ne J$ in $\mathcal G$. The intervals are disjoint, so $c_J\notin I$ and $|c_I-c_J|\ge\ell_I/2$. Since $c_I,c_J\in[0,1]$, $|c_I-c_J|\le1<\pi$ and $2\sin(y/2)\ge\frac2\pi y$ for $y\in[0,\pi]$, we have that
\[
    |\zeta_I-\zeta_J|=2\sin\frac{|c_I-c_J|}2\ge\frac2\pi|c_I-c_J|\ge\frac{\ell_I}\pi .
\]
Consequently, for every $J\in\mathcal G$,
\[
    \sum_{I\ne J}\frac{s_I}{|\zeta_I-\zeta_J|}\le\pi\sum_{I\ne J}\frac{\ell_I^2}{\ell_I}\le\pi\sum_{I\in\mathcal G}\ell_I=\pi .
\]
So \eqref{assumpsummab} holds with $M\le\pi$, and Theorem \ref{mainthmcountablespt} gives $\mathcal{D}(\mu)=\mathcal M(\bar a)=\mathcal{H}(b)$.

\end{proof}

\section{Measures with atoms at accumulation points}\label{S7}

In the previous section, the atoms of $\mu$ were isolated points of the support. Here we add to such a measure finitely many atoms \emph{placed at accumulation points of the other atoms}.
Write the resulting measure as
$\mu=\nu+\sum_{\ell=1}^N c_\ell\delta_{\eta_\ell}$,
where $\nu$ is the original measure.
The measure $\mu$ no longer satisfies the hypotheses of Theorem \ref{mainthmcountablespt}, but the construction can be reduced to that setting by replacing $\nu$ with the weighted measure $|P|^2\nu$, where $P$ vanishes at the new atoms, see Lemma \ref{potenvmucpvnutl}.
A new phenomenon appears: $\mathcal{D}(\mu)$ is still a space $\mathcal{H}(b)$, but $b$ can no longer be of the form $(1+u)/2$ with $u$ inner, as it was in Theorem \ref{mainthmcountablespt}. We begin with a lemma. 
\begin{lemma}\label{potenvmucpvnutl}
Let $\nu$ be a finite positive Borel measure on $\T$. For $N\geq 1$, take distinct points $\eta_1,\dots,\eta_N\in\T$ and masses $c_1,\dots,c_N>0$. Set
\[
    \mu=\nu+\sum_{\ell=1}^Nc_\ell\,\delta_{\eta_\ell},\qquad P(z)=\prod_{\ell=1}^N(1-\overline{\eta_\ell}z),\qquad \operatorname{d}\!\tilde\nu(\xi)=|P(\xi)|^2\operatorname{d}\!\nu(\xi).
\]
Then,
\begin{equation}\label{pvmuvnutilta}
    |P(z)|^2\bigl(1+V_\mu(z)\bigr)\asymp1+V_{\tilde\nu}(z),\qquad z\in\mathbb D .
\end{equation}
\end{lemma}

\begin{proof}
For $\ell=1,\ldots,N$, consider the polynomial $R_\ell(z)=P(z)/(1-\overline{\eta_\ell}z)=\prod_{j\ne\ell}(1-\overline{\eta_j}z)$. The family of polynomials $\{P,R_1,\ldots, R_N\}$ does not have a common zero. In particular, 
\[
|P(z)|^2+\sum_{\ell=1}^Nc_\ell|R_\ell(z)|^2\asymp1,\qquad z\in\overline{\mathbb D}.
\]
Since $V_\mu=\sum_{\ell=1}^N c_\ell|z-\eta_\ell|^{-2}+V_\nu$, this gives
\begin{align}
\notag |P(z)|^2\bigl(1+V_\mu(z)\bigr)&=|P(z)|^2+\sum_{\ell=1}^Nc_\ell|R_\ell(z)|^2+|P(z)|^2V_\nu(z)\\
\label{pvumcompvnu}& \asymp1+|P(z)|^2V_\nu(z).
\end{align}

Now, let $L:=\sup_{\overline{\mathbb D}}|P'|<\infty$. Then $|P(\xi)-P(z)|\le L|\xi-z|$ for $\xi,z\in\overline{\mathbb D}$. Moreover,
\[
\begin{aligned}
V_{\tilde\nu}(z)&=\int_{\T}\frac{|P(\xi)|^2}{|\xi-z|^2}\operatorname{d}\!\nu(\xi)\\
&\le2\int_{\T}\frac{|P(\xi)-P(z)|^2}{|\xi-z|^2}\operatorname{d}\!\nu(\xi)+2|P(z)|^2\int_{\T}\frac{\operatorname{d}\!\nu(\xi)}{|\xi-z|^2}\\
&\le2L^2 \nu(\T)+2|P(z)|^2V_\nu(z).
\end{aligned}
\]
In the same way, exchanging the roles of $P(\xi)$ and $P(z)$,
\[
    |P(z)|^2V_\nu(z)\le2V_{\tilde\nu}(z)+2L^2\nu(\T).
\]
Hence $1+|P|^2V_\nu\asymp1+V_{\tilde\nu}$, and together with \eqref{pvumcompvnu} this proves \eqref{pvmuvnutilta}.
\end{proof}

We specialize the previous lemma to the measures considered in Section \ref{S6}. In particular, for the rest of the section we will fix the measure $\nu$ given by $\nu=\sum_{n\ge1}s_n^2\,\delta_{\zeta_n}$, with $\zeta_n\in\T$ distinct and $s_n>0$ such that $\sum_ns_n^2<\infty$. We denote $E$ the support of $\nu$. 

For $N\in\mathbb{N}$, we consider distinct points $\eta_1,\dots,\eta_N$ in $E\setminus\{\zeta_n:n\ge1\}$ and masses $c_1,\ldots,c_N>0$. We set $P(z)=\prod_{\ell=1}^N(1-\overline{\eta_\ell}z)$. We define a new measure via
\begin{equation}\label{eq:defmusection8}
    \mu=\nu+\sum_{\ell=1}^Nc_\ell\,\delta_{\eta_\ell}.
\end{equation}
    
The goal of the next theorem is to give a sufficient condition that ensures that $\mathcal{D}(\mu)$ coincides with a de Branges--Rovnyak space.

\begin{theorem}\label{mainthmcounsprtaccpt} Consider a measure $\nu=\sum_{n\ge1}s_n^2\,\delta_{\zeta_n}$ as above, $N\ge1$, distinct points $\eta_1,\dots,\eta_N$ in $E\setminus\{\zeta_n:n\ge1\}$, masses $c_1,\ldots,c_N>0$, and the polynomial $P$ as above. Consider the measure $\mu$ as in \eqref{eq:defmusection8}, and set
\[
    t_n:=s_n|P(\zeta_n)|=s_n\prod_{\ell=1}^N|\zeta_n-\eta_\ell|,\quad n\geq 1,\qquad\tilde\nu=\sum_{n=1}^\infty t_n^2\,\delta_{\zeta_n}.
\]
 
Suppose that
\begin{equation}\label{summabiltywatap}
    M:=\sup_n\sum_{k\ne n}\frac{t_k}{|\zeta_k-\zeta_n|}<\infty .
\end{equation}

Finally, set
\[
    G(z)=\sum_nt_n\frac{\zeta_n+z}{\zeta_n-z},\qquad a(z)=\frac{P(z)}{2^{N+1}\bigl(1+G(z)\bigr)} .
\]
Then $a$ is outer, $a(0)>0$ and $\|a\|_\infty\le\frac12$. Let $b$ be the Pythagorean mate of $a$. Then
\[
    \mathcal{D}(\mu)=\mathcal M(\bar a)=\mathcal{H}(b).
\]
Moreover,
\[
    \mathcal{D}(\mu)\ne\mathcal H\bigl((1+u)/2\bigr)\qquad\text{for every inner function }u .
\]
\end{theorem}

\begin{remark}\label{remaccptsupport}
The measure $\mu$ is concentrated on a countable set, but its support~$E$ (note that $\eta_\ell\in E$) can be uncountable. For instance, take for $\nu$ the measure of Proposition \ref{uncounsptcm} and for $\eta_\ell$ distinct points of $\{e^{it}:t\in C\}$; these are accumulation points of the atoms, and they are not atoms since the midpoints $c_I$ lie in the removed intervals, not in $C$.
Moreover, since $\operatorname{diam}(\overline{\mathbb{D}})=2$, we have $t_n\le2^Ns_n$, and so \eqref{summabiltywatap} follows from \eqref{assumpsummab} for $\nu$: in that example the theorem applies.
\end{remark}

Before the proof of Theorem \ref{mainthmcounsprtaccpt}, we present the general result that will allow us to conclude that the space $\mathcal{D}(\mu)$ that we are considering cannot be equal to a de Branges--Rovnyak space $\mathcal{H}(b)$ with $b=(1+u)/2$. Given a set $A\subseteq\mathbb{C}$ we denote by $A'$ the set of accumulation points of $A$, that is, the set of points $z\in\mathbb{C}$ such that every \emph{pierced} neighbourhood of $z$ intersects $A$. Notice that in general $A$ is not a subset of $A'$.

\begin{proposition}\label{dmunehupohf}
Let $\omega$ be a finite positive measure on $\T$.  Suppose that there exists $u\not\equiv1$ inner such that 
\[
    \HD(\omega)=\mathcal H\bigl((1+u)/2\bigr).
\]
Denote by $A_{\omega}$ the set of atoms of $\omega$, and by $A_\omega'$ the set of its accumulation points. Then, $\omega(\{\eta\})=0$ for every $\eta\in A_\omega'$. In particular, $A_\omega\cap A_\omega'=\emptyset.$
\end{proposition}

\begin{proof}
Notice that
\[
V_\omega(z) = \int_{\mathbb{T}}\frac{\operatorname{d}\!\omega(\xi)}{|z-\xi|^2}\geq \frac{\omega(\mathbb{T})}{4}, \qquad z\in\overline{\mathbb{D}}.
\]
In particular, $1+V_\omega\asymp V_\omega$. By Corollary \ref{cordmuhohfinniffc},
\begin{equation*}
    |1-u(z)|^2V_\omega(z)\asymp1,\qquad z\in\mathbb D .
\end{equation*}
For any atom $\zeta\in A_\omega$  with mass $c_\zeta=\omega(\{\zeta\})>0$ we have $V_\omega(z)\ge c_\zeta/|z-\zeta|^2$, hence
\begin{equation}\label{u1estatom}
    |1-u(z)|\le C_\zeta\,|z-\zeta|,\qquad z\in\mathbb D.
\end{equation}
Suppose by contradiction that $\eta\in A_\omega\cap A_\omega'$.
Apply \eqref{u1estatom} to $\eta$, and choose $r>0$ with $C_\eta r<\frac12$. Then, on $\mathbb D\cap B(\eta,r)$ we have that $|1-u|<\frac12$, hence $|u|>\frac12$. This implies that $u$ can be extended analytically on a neighbourhood $W$ of $\eta$ (this is related to the fact that $\eta$ is not in the \emph{spectrum} of $u$, see for example \cite[Subsection 7.3]{GMR16}).

Since $\eta\in A_\omega'$, there exist infinitely many atoms $\xi_k$ in $W$. By \eqref{u1estatom} with $\zeta=\xi_k$ and letting $z\to\xi_k$ inside $\mathbb D$, we get $u(\xi_k)=1$. So the analytic function $u-1$ on $W$ has infinitely many distinct zeros $\xi_k$ accumulating at the interior point $\eta$ of $W$. This implies that $u\equiv1$, a contradiction.
\end{proof}

We prove the main theorem of this section.

\begin{proof}[Proof of Theorem \ref{mainthmcounsprtaccpt}]
We begin by considering the measure $\tilde\nu=\sum_{n=1}^\infty t_n^2\,\delta_{\zeta_n}$, where 
\(    t_n=s_n|P(\zeta_n)|, n\geq 1.
\) The weights $t_n$ are positive because $\zeta_n\ne\eta_\ell$ for all $n,\ell$, and $\sum_nt_n^2\le4^N\sum_ns_n^2<\infty$.  Condition \eqref{summabiltywatap} is precisely condition \eqref{assumpsummab} for $\tilde\nu$. We remark that, as in Lemma \ref{lemconseq}(a), this condition implies $\sum_nt_n<\infty$. Then, we can apply Theorem~\ref{mainthmcountablespt} to $\tilde\nu$. In particular, the function $a_0(z)=1/(1+G(z))$ is outer, $a_0(0)>0$, $\|a_0\|_\infty\le1$, $\HD(\tilde\nu)=\mathcal M(\overline{a_0})$ and $|a_0(z)|^2\bigl(1+V_{\tilde\nu}(z)\bigr)\asymp1$ uniformly on  $z\in\mathbb D$.

Since $a=2^{-N-1}Pa_0$, Lemma \ref{potenvmucpvnutl} yields
\[
\begin{aligned}
    |a(z)|^2\bigl(1+V_\mu(z)\bigr)
    &=2^{-2N-2}|a_0(z)|^2|P(z)|^2\bigl(1+V_\mu(z)\bigr)\\
    &\asymp |a_0(z)|^2\bigl(1+V_{\tilde\nu}(z)\bigr)
    \asymp1.
\end{aligned}
\]
The polynomial $P$ is a product of outer factors $1-\overline{\eta_\ell}z$, so $a$ is outer. Also,
\[
    a(0)=\frac1{2^{N+1}(1+\sum_{n\geq 1}t_n)}>0,\qquad
    \|a\|_\infty\le2^{-N-1}\,2^N=\frac12.
\]
Let $b$ be the normalized outer Pythagorean mate of $a$. Then $|b|^2=1-|a|^2\ge3/4$ a.e. on $\T$. Taking Poisson integrals of $\log|b|$ gives $|b(z)|\ge\sqrt3/2$ for every $z\in\mathbb D$. Thus $(b,a)$ is a Corona pair, and Corollary \ref{maincordmuqhbiffc2} gives
\[
    \mathcal{D}(\mu)=\mathcal M(\bar a)=\mathcal{H}(b).
\]
Finally, Proposition \ref{dmunehupohf} excludes every representation $\mathcal{D}(\mu)=\mathcal H((1+u)/2)$ with $u$ inner.
\end{proof}

\section*{Methodology}
The use of AI tools for exploratory discussions during the preparation of this work is acknowledged. More specifically, the Corona condition of $(b,a)$ in Theorem \ref{T:main1} and Theorem \ref{T:main2} comes from interaction of the authors with the LLMs. All mathematical statements, proofs, and references were independently checked, revised, and finalized by the authors, who take full responsibility for the content of the paper. 

\section*{Acknowledgments}
The first author thanks Iorgos Nikolaidis for helpful discussions on some preliminary ideas. The first two authors are
grateful to Andreas Hartmann and Javad Mashreghi for many stimulating discussions and for their generosity in sharing their mathematical insight. Their availability has been fundamental for our mathematical growth, and we thank them warmly.

\section*{Grant}
The first two authors are members of Gruppo Nazionale per l'Analisi Matematica, la Probabilità e le loro Applicazioni (GNAMPA) of Istituto Nazionale di Alta Matematica (INdAM).
The third author was supported by National Science Foundation of China (12271149), Natural Science Foundation of Hunan Province (2024JJ2008).

\bibliographystyle{abbrv}
\bibliography{biblio}

\vspace{22 pt}

%
%

\end{document}